\documentclass{amsart}

\usepackage{amssymb,amsmath,stmaryrd,mathrsfs}
\def\definetac{\newif\iftac}    
\ifx\tactrue\undefined
  \definetac
  \ifx\state\undefined\tacfalse\else\tactrue\fi
\fi
\iftac\else\usepackage{amsthm}\fi
\usepackage[all,2cell]{xy}
\UseAllTwocells
\usepackage{enumitem}
\usepackage{xcolor}
\definecolor{darkgreen}{rgb}{0,0.45,0} 
\usepackage[pagebackref,colorlinks,citecolor=darkgreen,linkcolor=darkgreen]{hyperref}
\usepackage{mathtools}          
\usepackage{tikz}
\usepackage{braket}             

\usepackage{url}                
\usepackage{xspace}             

\makeatletter
\let\ea\expandafter

\def\mdef#1#2{\ea\ea\ea\gdef\ea\ea\noexpand#1\ea{\ea\ensuremath\ea{#2}\xspace}}
\def\alwaysmath#1{\ea\ea\ea\global\ea\ea\ea\let\ea\ea\csname your@#1\endcsname\csname #1\endcsname
  \ea\def\csname #1\endcsname{\ensuremath{\csname your@#1\endcsname}\xspace}}

\DeclareRobustCommand\widecheck[1]{{\mathpalette\@widecheck{#1}}}
\def\@widecheck#1#2{%
    \setbox\z@\hbox{\m@th$#1#2$}%
    \setbox\tw@\hbox{\m@th$#1%
       \widehat{%
          \vrule\@width\z@\@height\ht\z@
          \vrule\@height\z@\@width\wd\z@}$}%
    \dp\tw@-\ht\z@
    \@tempdima\ht\z@ \advance\@tempdima2\ht\tw@ \divide\@tempdima\thr@@
    \setbox\tw@\hbox{%
       \raise\@tempdima\hbox{\scalebox{1}[-1]{\lower\@tempdima\box
\tw@}}}%
    {\ooalign{\box\tw@ \cr \box\z@}}}

\newcount\foreachcount

\def\foreachletter#1#2#3{\foreachcount=#1
  \ea\loop\ea\ea\ea#3\@alph\foreachcount
  \advance\foreachcount by 1
  \ifnum\foreachcount<#2\repeat}

\def\foreachLetter#1#2#3{\foreachcount=#1
  \ea\loop\ea\ea\ea#3\@Alph\foreachcount
  \advance\foreachcount by 1
  \ifnum\foreachcount<#2\repeat}

\def\definescr#1{\ea\gdef\csname s#1\endcsname{\ensuremath{\mathscr{#1}}\xspace}}
\foreachLetter{1}{27}{\definescr}
\def\definecal#1{\ea\gdef\csname c#1\endcsname{\ensuremath{\mathcal{#1}}\xspace}}
\foreachLetter{1}{27}{\definecal}
\def\definebold#1{\ea\gdef\csname b#1\endcsname{\ensuremath{\mathbf{#1}}\xspace}}
\foreachLetter{1}{27}{\definebold}
\def\definebb#1{\ea\gdef\csname l#1\endcsname{\ensuremath{\mathbb{#1}}\xspace}}
\foreachLetter{1}{27}{\definebb}
\def\definefrak#1{\ea\gdef\csname f#1\endcsname{\ensuremath{\mathfrak{#1}}\xspace}}
\foreachletter{1}{9}{\definefrak} 
\foreachletter{10}{27}{\definefrak}
\def\definebar#1{\ea\gdef\csname #1bar\endcsname{\ensuremath{\overline{#1}}\xspace}}
\foreachLetter{1}{27}{\definebar}
\foreachletter{1}{8}{\definebar} 
\foreachletter{9}{15}{\definebar} 
\foreachletter{16}{27}{\definebar}
\def\definetil#1{\ea\gdef\csname #1til\endcsname{\ensuremath{\widetilde{#1}}\xspace}}
\foreachLetter{1}{27}{\definetil}
\foreachletter{1}{27}{\definetil}
\def\definehat#1{\ea\gdef\csname #1hat\endcsname{\ensuremath{\widehat{#1}}\xspace}}
\foreachLetter{1}{27}{\definehat}
\foreachletter{1}{27}{\definehat}
\def\definechk#1{\ea\gdef\csname #1chk\endcsname{\ensuremath{\widecheck{#1}}\xspace}}
\foreachLetter{1}{27}{\definechk}
\foreachletter{1}{27}{\definechk}
\def\defineul#1{\ea\gdef\csname u#1\endcsname{\ensuremath{\underline{#1}}\xspace}}
\foreachLetter{1}{27}{\defineul}
\foreachletter{1}{27}{\defineul}

\def\autofmt@n#1\autofmt@end{\mathrm{#1}}
\def\autofmt@b#1\autofmt@end{\mathbf{#1}}
\def\autofmt@l#1#2\autofmt@end{\mathbb{#1}\mathsf{#2}}
\def\autofmt@c#1#2\autofmt@end{\mathcal{#1}\mathit{#2}}
\def\autofmt@s#1#2\autofmt@end{\mathscr{#1}\mathit{#2}}
\def\autofmt@f#1\autofmt@end{\mathsf{#1}}
\def\autofmt@u#1\autofmt@end{\underline{\smash{\mathsf{#1}}}}
\def\autofmt@U#1\autofmt@end{\underline{\underline{\smash{\mathsf{#1}}}}}
\def\autofmt@h#1\autofmt@end{\widehat{#1}}
\def\autofmt@r#1\autofmt@end{\overline{#1}}
\def\autofmt@t#1\autofmt@end{\widetilde{#1}}
\def\autofmt@k#1\autofmt@end{\check{#1}}

\def\auto@drop#1{}
\def\autodef#1{\ea\ea\ea\@autodef\ea\ea\ea#1\ea\auto@drop\string#1\autodef@end}
\def\@autodef#1#2#3\autodef@end{%
  \ea\def\ea#1\ea{\ea\ensuremath\ea{\csname autofmt@#2\endcsname#3\autofmt@end}\xspace}}
\def\autodefs@end{blarg!}
\def\autodefs#1{\@autodefs#1\autodefs@end}
\def\@autodefs#1{\ifx#1\autodefs@end%
  \def\autodefs@next{}%
  \else%
  \def\autodefs@next{\autodef#1\@autodefs}%
  \fi\autodefs@next}

\DeclareSymbolFont{bbold}{U}{bbold}{m}{n}
\DeclareSymbolFontAlphabet{\mathbbb}{bbold}

\newcommand{\bbone}{\ensuremath{\mathbbb{1}}\xspace}

\mdef\delbar{\overline{\partial}}

\mdef\hf{\textstyle\frac12 }
\mdef\thrd{\textstyle\frac13 }
\mdef\qtr{\textstyle\frac14 }

\newcommand{\op}{^{\mathrm{op}}}

\SelectTips{cm}{}
\newdir{ >}{{}*!/-10pt/@{>}}    
\newcommand{\pushoutcorner}[1][dr]{\save*!/#1+1.2pc/#1:(1,-1)@^{|-}\restore}

\mdef\Id{\mathrm{Id}}
\mdef\id{\mathrm{id}}
\alwaysmath{ell}
\alwaysmath{infty}
\alwaysmath{odot}
\def\frc#1/#2.{\frac{#1}{#2}}   
\mdef\ten{\mathrel{\otimes}}

\mdef\sqten{\mathrel{\boxtimes}}

\DeclareRobustCommand\widecheck[1]{{\mathpalette\@widecheck{#1}}}
\def\@widecheck#1#2{%
    \setbox\z@\hbox{\m@th$#1#2$}%
    \setbox\tw@\hbox{\m@th$#1%
       \widehat{%
          \vrule\@width\z@\@height\ht\z@
          \vrule\@height\z@\@width\wd\z@}$}%
    \dp\tw@-\ht\z@
    \@tempdima\ht\z@ \advance\@tempdima2\ht\tw@ \divide\@tempdima\thr@@
    \setbox\tw@\hbox{%
       \raise\@tempdima\hbox{\scalebox{1}[-1]{\lower\@tempdima\box
\tw@}}}%
    {\ooalign{\box\tw@ \cr \box\z@}}}

\DeclareMathOperator\colim{colim}

\mdef\we{\overset{\sim}{\longrightarrow}}
\mdef\leftwe{\overset{\sim}{\longleftarrow}}

\let\xto\xrightarrow

\def\rightarrowtailfill@{\arrowfill@{\Yright\joinrel\relbar}\relbar\rightarrow}
\newcommand\xrightarrowtail[2][]{\ext@arrow 0055{\rightarrowtailfill@}{#1}{#2}}

\def\twoheadrightarrowfill@{\arrowfill@{\relbar\joinrel\relbar}\relbar\twoheadrightarrow}
\newcommand\xtwoheadrightarrow[2][]{\ext@arrow 0055{\twoheadrightarrowfill@}{#1}{#2}}

\def\slashedarrowfill@#1#2#3#4#5{%
  $\m@th\thickmuskip0mu\medmuskip\thickmuskip\thinmuskip\thickmuskip
   \relax#5#1\mkern-7mu%
   \cleaders\hbox{$#5\mkern-2mu#2\mkern-2mu$}\hfill
   \mathclap{#3}\mathclap{#2}%
   \cleaders\hbox{$#5\mkern-2mu#2\mkern-2mu$}\hfill
   \mkern-7mu#4$%
}
\def\rightslashedarrowfill@{%
  \slashedarrowfill@\relbar\relbar\mapstochar\rightarrow}
\newcommand\xslashedrightarrow[2][]{%
  \ext@arrow 0055{\rightslashedarrowfill@}{#1}{#2}}
\mdef\hto{\xslashedrightarrow{}}
\mdef\htoo{\xslashedrightarrow{\quad}}

\def\toiso{\xto{\smash{\raisebox{-.5mm}{$\scriptstyle\sim$}}}}

\long\def\my@drawfill#1#2;{%
\@skipfalse
\fill[#1,draw=none] #2;
\@skiptrue
\draw[#1,fill=none] #2;
}
\newif\if@skip
\newcommand{\skipit}[1]{\if@skip\else#1\fi}
\newcommand{\drawfill}[1][]{\my@drawfill{#1}}

\newif\ifhyperref
\@ifpackageloaded{hyperref}{\hyperreftrue}{\hyperreffalse}
\iftac
  \let\your@state\state
  \def\state#1{\gdef\currthmtype{#1}\your@state{#1}}
  \let\your@staterm\staterm
  \def\staterm#1{\gdef\currthmtype{#1}\your@staterm{#1}}
  \let\defthm\newtheorem
  \def\currthmtype{}
  \ifhyperref
    \def\autoref#1{\ref*{label@name@#1}~\ref{#1}}
  \else
    \def\autoref#1{\ref{label@name@#1}~\ref{#1}}
  \fi
  \AtBeginDocument{%
    \let\old@label\label%
    \def\label#1{%
      {\let\your@currentlabel\@currentlabel%
        \edef\@currentlabel{\currthmtype}%
        \old@label{label@name@#1}}%
      \old@label{#1}}
  }
\else
  \ifhyperref
    \def\defthm#1#2{%
      \newtheorem{#1}{#2}[section]%
      \expandafter\def\csname #1autorefname\endcsname{#2}%
      \expandafter\let\csname c@#1\endcsname\c@thm}
  \else
    \def\defthm#1#2{\newtheorem{#1}[thm]{#2}}
    \ifx\SK@label\undefined\let\SK@label\label\fi
    \let\old@label\label
    \let\your@thm\@thm
    \def\@thm#1#2#3{\gdef\currthmtype{#3}\your@thm{#1}{#2}{#3}}
    \def\currthmtype{}
    \def\label#1{{\let\your@currentlabel\@currentlabel\def\@currentlabel%
        {\currthmtype~\your@currentlabel}%
        \SK@label{#1@}}\old@label{#1}}
    \def\autoref#1{\ref{#1@}}
  \fi
\fi

\newtheorem{thm}{Theorem}[section]

\defthm{cor}{Corollary}
\defthm{prop}{Proposition}
\defthm{lem}{Lemma}
\defthm{sch}{Scholium}
\defthm{assume}{Assumption}
\defthm{claim}{Claim}
\defthm{conj}{Conjecture}
\defthm{hyp}{Hypothesis}
\defthm{fact}{Fact}
\iftac\theoremstyle{plain}\else\theoremstyle{definition}\fi
\defthm{defn}{Definition}
\defthm{notn}{Notation}
\iftac\theoremstyle{plain}\else\theoremstyle{remark}\fi
\defthm{rmk}{Remark}
\defthm{eg}{Example}
\defthm{egs}{Examples}
\defthm{ex}{Exercise}
\defthm{ceg}{Counterexample}
\defthm{warn}{Warning}
\defthm{con}{Construction}

\def\thmqedhere{\expandafter\csname\csname @currenvir\endcsname @qed\endcsname}

\setitemize[1]{leftmargin=2em}
\setenumerate[1]{leftmargin=*}

\iftac
  \let\c@equation\c@subsection
\else
  \let\c@equation\c@thm
\fi
\numberwithin{equation}{section}

\@ifpackageloaded{mathtools}{\mathtoolsset{showonlyrefs,showmanualtags}}{}

\alwaysmath{alpha}
\alwaysmath{beta}
\alwaysmath{gamma}
\alwaysmath{Gamma}
\alwaysmath{delta}
\alwaysmath{Delta}
\alwaysmath{epsilon}
\mdef\ep{\varepsilon}
\alwaysmath{zeta}
\alwaysmath{eta}
\alwaysmath{theta}
\alwaysmath{Theta}
\alwaysmath{iota}
\alwaysmath{kappa}
\alwaysmath{lambda}
\alwaysmath{Lambda}
\alwaysmath{mu}
\alwaysmath{nu}
\alwaysmath{xi}
\alwaysmath{pi}
\alwaysmath{rho}
\alwaysmath{sigma}
\alwaysmath{Sigma}
\alwaysmath{tau}
\alwaysmath{upsilon}
\alwaysmath{Upsilon}
\alwaysmath{phi}
\alwaysmath{Pi}
\alwaysmath{Phi}
\mdef\ph{\varphi}
\alwaysmath{chi}
\alwaysmath{psi}
\alwaysmath{Psi}
\alwaysmath{omega}
\alwaysmath{Omega}

\makeatother

\tikzset{lab/.style={auto,font=\scriptsize}} 
\usetikzlibrary{arrows}

\usepackage{xcolor}
\usepackage{fixme}
\fxsetup{
    status=draft,
    author=,
    layout=margin,
    theme=color
}

\definecolor{fxnote}{rgb}{1.0000,0.0000,0.0000}
\colorlet{fxnotebg}{yellow}

\newcommand{\D}{\sD}
\newcommand{\E}{\sE}

\autodefs{\cDER\cCat\cGrp\cCAT\cMONCAT\cTriaCAT\cAddCAT\cPreCAT\cPtCAT\cSp\bSet\cProf\bCat
\cPro\cMONDER\cBICAT\ncomp\nid\ncoker\niso\ndia\nIm\sProf\ncyl\fC\fZ\fT\nhocolim\nholim\cPsNat\ncoll\nCh\bsSet\cAlg\cIm\cI\cP}

\let\oldboxtimes\boxtimes
\def\boxtimes{\mathrel{\oldboxtimes}}

\newcommand{\fib}{\mathsf{fib}}
\newcommand{\cof}{\mathsf{cof}}

\def\ccsub{_{\mathrm{cc}}}
\def\pdh(#1,#2){\llbracket #1,#2\rrbracket}
\def\ldh(#1,#2){\llbracket #1,#2\rrbracket\ccsub}
\def\pend(#1){\pdh(#1,#1)}
\def\lend(#1){\ldh(#1,#1)}

\def\DTl#1#2#3#4#5#6#7{%
  \xymatrix@C=3pc{{#1} \ar[r]^-{#2} &
    {#3} \ar[r]^-{#4} &
    {#5} \ar[r]^-{#6} &
    {#7}
  }}

\newsavebox{\tvabox}
\savebox\tvabox{\hspace{1mm}\begin{tikzpicture}[>=latex',baseline={(0,-.18)}]
  \draw[->] (0,.1) -- +(1,0);
  \node at (.5,0) {$\scriptscriptstyle\bot$};
  \draw[->] (1,-.1) -- +(-1,0);
  \draw[->] (1,-.2) -- +(-1,0);
\end{tikzpicture}\hspace{1mm}}

\renewcommand{\ex}{\mathrm{ex}}

\newtheorem*{thm*}{\textbf{Theorem}}

\title{Revisiting the canonicity of canonical triangulations}

\author{Moritz Groth}
\address{Rheinische Friedrich-Wilhelms-Universit{\"a}t Bonn, Mathematisches Institut, Ende-nicher Allee 60, 53115 Bonn, Germany}
\email{mgroth@math.uni-bonn.de}

\date{\today}

\begin{document}

\begin{abstract}
Stable derivators provide an enhancement of triangulated categories as is indicated by the existence of canonical triangulations. In this paper we show that exact morphisms of stable derivators induce exact functors of canonical triangulations, and similarly for arbitrary natural transformations. This $2$-categorical refinement also provides a uniqueness statement concerning canonical triangulations.

These results rely on a more careful study of morphisms of derivators and this study is of independent interest. We analyze the interaction of morphisms of derivators with limits, colimits, and Kan extensions, including a discussion of invariance and closure properties of the class of Kan extensions preserved by a fixed morphism. 
\end{abstract}

\maketitle

\tableofcontents

\section{Introduction}
\label{sec:intro}

Abstract stable homotopy theories arise in various areas of pure mathematics such as algebra, geometry, and topology. One typical class of examples is provided by homological algebra. More specifically, associated to a Grothendieck abelian category $\cA$ there is the stable homotopy theory $\nCh(\cA)$ of unbounded chain complexes in~$\cA$. Another typical example of an abstract stable homotopy theory is given by the stable homotopy theory \cSp of spectra in the sense of topology and, in a certain precise sense, this yields the universal example of an abstract stable homotopy theory \cite{franke:adams,heller:stable,HA}.

There are various ways of making precise what one means by an abstract stable homotopy theory, and one of the more classical approaches is provided by \emph{triangulated categories} as introduced by Verdier in \cite{verdier:thesis,verdier:derived} (see also \cite{puppe:stabil}). In the above two specific situations, this leads to Verdier's classical triangulations on derived categories $D(\cA)$ and to Boardman's classical triangulation on \emph{the} stable homotopy category $\mathcal{SHC}$ \cite{boardman:thesis,vogt:boardman}. The basic idea behind the \emph{structure} of a triangulation is that distinguished triangles 
\begin{equation}\label{eq:triang}
X\stackrel{f}{\to} Y\stackrel{g}{\to} Z\stackrel{h}{\to} \Sigma X
\end{equation}
on additive categories (such as $D(\cA)$ or $\mathcal{SHC}$) encode `certain shadows of iterated derived cokernel constructions on the models in the background (such as $\nCh(\cA)$ or \cSp)'.  And this idea is for instance made precise in \cite{groth:intro-to-der-1}.

Correspondingly, there is the notion of an \emph{exact functor} $F\colon\cT\to\cT'$ of triangulated categories, which is roughly defined as an additive functor which sends distinguished triangles to distinguished triangles. To make this precise, exact functors $\cT\to\cT'$ are defined as pairs consisting of an additive functor $F\colon\cT\to\cT'$ and a natural transformation
\begin{equation}\label{eq:intro-exact}
\sigma\colon F\Sigma\to\Sigma F
\end{equation}
such that for every distinguished triangle \eqref{eq:triang} in $\cT$ the image triangle
\[
FX\stackrel{Ff}{\to}FY\stackrel{Fg}{\to}FZ\stackrel{\sigma\circ Fh}{\to}\Sigma FX
\]
is distinguished in $\cT'$. It follows that the \emph{exact structure} $\sigma\colon F\Sigma\to\Sigma F$ is a natural isomorphism, and the existence of such an exact morphism expresses the idea that `on the models in the background there is a morphism of stable homotopy theories which preserves sufficiently finite derived or homotopy (co)limits'.

In this paper we make this second idea precise in the language of derivators \cite{grothendieck:derivators,heller:htpythies,franke:adams}. It is known that the values of a (strong) stable derivator can be turned into triangulated categories \cite{franke:adams,maltsiniotis:seminar,groth:ptstab}, and here we investigate the $2$-functoriality of these triangulations. More precisely, if $F\colon\D\to\E$ is a pointed morphism of pointed derivators, then there is a canonical natural transformation
\begin{equation}\label{eq:intro-can}
\psi\colon\Sigma F\to F\Sigma
\end{equation}
which is compatible with \emph{all} natural transformations of derivators. These canonical transformations \eqref{eq:intro-can} are invertible for \emph{right exact} morphisms, i.e., for morphisms which preserve initial objects and pushouts, and in the stable case the \emph{inverse} transformation $\sigma=\psi^{-1}$ is shown to define an exact structure \eqref{eq:intro-exact} with respect to canonical triangulations. 

Specializing this to identity morphisms it follows that canonical triangulations are unique in a certain precise sense, thereby justifying the terminology. Moreover, if we specialize our results to restriction morphisms and induced transformations, then this shows that (strong) stable derivators admit lifts to the $2$-category of triangulated categories, exact functors, and exact transformations. (It is fairly straightforward to adapt the techniques from this paper in order to establish variants for canonical \emph{higher} triangulations \cite{beilinson:perverse,maltsiniotis:higher,gst:Dynkin-A}.)

A detailed proof of these results turns out to be more lengthy than suggested by the sketch proof of the author in \cite{groth:ptstab}. Hence, by popular demand, such proofs are provided in this paper (a more conceptual perspective on such results will be provided elsewhere). Since many of the other enhancements of triangulated categories (such as stable cofibration categories \cite{schwede:p-order}, stable model categories \cite{hovey:model}, and stable $\infty$-categories \cite{HA}) have underlying homotopy derivators (at least of suitable types), the results obtained here also have implications for these other approaches. In the case of stable cofibration categories a $1$-categorical version of our results was established by Schwede in \cite{schwede:p-order}.

To the opinion of the author, the techniques developed here and leading to the above results are at least as interesting as the results themselves. In \S\S\ref{sec:colimiting-cocones}-\ref{sec:exact-htpy-finite} we collect various tools related to morphisms and natural transformations of derivators which we need for our applications here and which will also prove useful elsewhere (for example in the sequel \cite{groth:char}). While some facts are already spread out in the literature, to the best of the knowledge of the author most of these results have not appeared elsewhere. In the few cases where we reprove a known result, the author claims that the proof given here is simpler than the existing one(s). The reader who prefers to take for granted the existence of a well-behaved formalism of morphisms and natural transformations of derivators is suggested to first focus on \S\S\ref{sec:pointed}-\ref{sec:can-triang} only.

One of the goals in \S\S\ref{sec:colimiting-cocones}-\ref{sec:exact-htpy-finite} is to study in more detail the interaction of morphisms and natural transformations of derivators with limits, colimits, and Kan extensions. This includes a discussion of invariance and closure properties of classes of colimits and left Kan extensions preserved by fixed morphisms of derivators. It turns out that the interaction with left Kan extensions along fully faithful functors is conceptually simpler than the general case. Hence, in this paper we stress that a discussion of the interaction with colimits can be reduced to the conceptually simpler situation by means of the cocone construction. 

To mention an additional specific result, we recall from \cite{ps:linearity} that a right exact morphism of derivators, i.e., a morphism which preserves initial objects and pushouts, already preserves homotopy finite colimits. Having established the basic theory of morphisms of derivators, it is straightforward to conclude that right exact morphisms also preserve left homotopy finite left Kan extensions. The point of this result is that many constructions arising in nature are combinations of such Kan extensions (see for example \cite{gst:basic,gst:tree,gst:Dynkin-A,gst:acyclic}). 

This paper belongs to a project aiming for an abstract study of stability, and the paper can be thought of as a sequel to \cite{groth:ptstab,gps:mayer} and as a prequel to \cite{groth:char}. This abstract study of stability was developed further in the series of papers on abstract representation theory \cite{gst:basic,gst:tree,gst:Dynkin-A,gst:acyclic} and this will be continued in \cite{gst:spectral-serre}. 

The content of the sections is as follows. In \S\ref{sec:colimiting-cocones} we study colimiting cocones in derivators. In \S\S\ref{sec:ctns-mor}-\ref{sec:ctns-closure} we discuss colimit preserving morphisms and establish some closure properties. In \S\ref{sec:paras} we study the compatibility of colimit preserving morphisms and the passage to parametrized versions. In \S\ref{sec:coproducts} we show that coproducts and homotopy coproducts agree in derivators and we obtain a similar result for morphisms of derivators. In \S\ref{sec:pointed} we note that pointed morphisms allow for canonical comparison maps related to suspensions, cones, cofiber sequences, and similar constructions, which we show in \S\ref{sec:exact-htpy-finite} to be invertible for right exact morphisms. In \S\ref{sec:exact-htpy-finite} we also deduce that right exact morphisms preserve left homotopy finite left Kan extensions. In \S\ref{sec:can-triang} we show that exact morphisms of strong stable derivators yield exact functors between canonical (higher) triangulations and similarly for arbitrary natural transformations. This leads to $2$-functoriality and uniqueness results for canonical (higher) triangulations. 

\textbf{Prerequisites.} We assume that the reader is familiar with the \emph{basic} language of derivators. Derivators were introduced independently by Grothendieck~\cite{grothendieck:derivators}, Heller~\cite{heller:htpythies}, and Franke~\cite{franke:adams}, and were developed further by various mathematicians including Maltsiniotis \cite{maltsiniotis:seminar,maltsiniotis:k-theory,maltsiniotis:htpy-exact} and Cisinski \cite{cisinski:direct,cisinski:loc-min,cisinski:derived-kan} (see \cite{grothendieck:derivators} for many additional references). In this paper we continue using the notation and conventions from \cite{gps:mayer} which together with \cite{groth:ptstab} provides some basic background. In particular, the axioms of a derivator are referred to by the names (Der1) (`coproducts are sent to products'), (Der2) (`isomorphisms are pointwise'), (Der3) (`left and right Kan extensions exist'), and (Der4) (`pointwise formulas for Kan extensions'). For a more detailed account of the basic theory we refer the reader to \cite{groth:intro-to-der-1}.

\section{Colimiting cocones in derivators}
\label{sec:colimiting-cocones}

In this section we discuss colimiting cocones in derivators and the construction of canonical comparison maps from colimiting cocones to arbitrary cocones. These notions and results will be used in later sections and also in the sequel \cite{groth:char}.

The \textbf{cocone} $A^\rhd$ of a small category $A$ is obtained by adjoining a new final object $\infty\in A^\rhd$ to $A$, while the \textbf{cone} $A^\lhd$ contains a new initial object $-\infty\in A^\lhd$. Related to these categories there are obvious fully faithful inclusion functors
\[
i_A\colon A\to A^\rhd\qquad\text{and}\qquad i_A\colon A\to A^\lhd.
\]

\begin{defn}\label{defn:colim-cocone}
Let \D be a derivator and let $A\in\cCat$. 
\begin{enumerate}
\item A diagram $X\in\D(A^\rhd)$ is a \textbf{cocone} (with \textbf{base} $i_A^\ast X\in\D(A)$). 
\item A cocone $X\in\D(A^\rhd)$ is \textbf{colimiting} if it lies in the essential image of $(i_A)_!\colon\D(A)\to\D(A^\rhd).$
\end{enumerate}
\end{defn}

\textbf{Cones} and \textbf{limiting cones} in a derivator are defined dually, and the following \textbf{duality principle} allows us often to focus on (colimiting) cocones. 

\begin{lem}\label{lem:duality-cones}
Let $\D\colon\cCat\op\to\cCAT$ be a derivator and let $A$ be a small category. A cocone $X\in\D(A^\rhd)$ is colimiting if and only if the cone $X\in\D\op((A\op)^\lhd)$ is limiting.
\end{lem}
\begin{proof}
This is immediate from the definitions.
\end{proof}

Recall that the calculus of Kan extensions in derivators is governed by the formalism of \textbf{homotopy exact squares}; see \cite{maltsiniotis:htpy-exact} or \cite{groth:ptstab}.

\begin{prop}\label{prop:htpy-ex-cocone}
For $A\in\cCat$ the following squares are homotopy exact,
\begin{equation}\label{eq:htpy-exact-cocone}
\vcenter{
\xymatrix{
A\ar[r]^-\id\ar[d]_-{\pi_A}\drtwocell\omit{}&A\ar[d]^-{i_A}&&
A\ar[r]^-\id\ar[d]_-{\pi_A}&A\ar[d]^-{i_A}\\
\bbone\ar[r]_-\infty&A^\rhd,&&
\bbone\ar[r]_-{-\infty}&A^\lhd.\ultwocell\omit{}
}
}
\end{equation}
\end{prop}
\begin{proof}
As slice squares these squares are homotopy exact by (Der4).
\end{proof}

More explicitly, the proposition says that for a derivator \D and $X\in\D(A)$ there are canonical isomorphisms
\[
\colim_AX\toiso(i_A)_!(X)_\infty\qquad\text{and}\qquad (i_A)_\ast(X)_{-\infty}\toiso\mathrm{lim}_A X.
\]
Such isomorphisms characterize colimiting cocones and limiting cones, as made precise by the following proposition, in which we consider the squares
\begin{equation}\label{eq:htpy-exact-cocone-2}
\vcenter{
\xymatrix{
A\ar[r]^{i_A}\ar[d]_{\pi_A} \drtwocell\omit &A^\rhd\ar[d]^\id&&
A\ar[r]^{i_A}\ar[d]_{\pi_A}  &A^\lhd\ar[d]^\id\\
\bbone\ar[r]_{\infty}&A^\rhd,&&
\bbone\ar[r]_{-\infty}&A^\lhd.\ultwocell\omit
}
}
\end{equation}

\begin{prop}\label{prop:cone}
Let \D be a derivator and let $A\in\cCat$.
The left Kan extension functor $(i_A)_!\colon\D(A)\to\D(A^\rhd)$ is fully faithful and $X\in\D(A^\rhd)$ lies in the essential image if and only if the canonical mate
\begin{equation}\label{eq:cone-II}
\colim_A i_A^\ast X\to X_\infty
\end{equation}
associated to the left square in \eqref{eq:htpy-exact-cocone-2} is an isomorphism.
\end{prop}
\begin{proof}
Since $i_A$ is fully faithful, so is $(i_A)_!\colon\D(A)\to\D(A^\rhd)$ and the essential image consists by \cite[Lem.~1.21]{groth:ptstab} precisely of those $X$ such that $\varepsilon\colon (i_A)_! i_A^\ast(X)\to X$ is an isomorphism at $\infty$. Note that the left square in \eqref{eq:htpy-exact-cocone-2} can also be written as the pasting
\[
\xymatrix{
A\ar[r]^-\id\ar[d]_-{\pi_A}\drtwocell\omit{}&A\ar[r]^-{i_A}\ar[d]_-{i_A}\drtwocell\omit{\id}&A^\rhd\ar[d]^-\id\\
\bbone\ar[r]_-\infty&A^\rhd\ar[r]_-\id&A^\rhd.
}
\]
Since the left square in this pasting is homotopy exact (\autoref{prop:htpy-ex-cocone}), the functoriality of mates with pasting allows us to conclude that $X$ lies in the essential image of $(i_A)_!$ if and only if the canonical mate \eqref{eq:cone-II} is an isomorphism.
\end{proof}

Thus, a cocone is colimiting in the sense of \autoref{defn:colim-cocone} if and only if the apex of it is canonically the colimit of the restriction to the base, justifying the terminology. 

The canonical morphism \eqref{eq:cone-II} also admits a different description which is inspired by the following trivial observation from ordinary category theory. Let \cC be a cocomplete category and let $G\colon A^\rhd\to\cC$ be a cocone on $F=Gi_A\colon A\to\cC$. It is immediate from the definition of a colimit as an initial cocone, that there is always a comparison map from the colimiting cocone on $F$ to~$G$ and that this comparison map is an isomorphism if and only if~$G$ is a colimiting cocone. To extend this to derivators we make the following construction.

\begin{con}\label{con:mor-of-cones}
We note that the cocone construction $A\mapsto A^\rhd$ is functorial, thereby defining $(-)^\rhd\colon\cCat\to\cCat$, and that the fully faithful functors $i_A\colon A\to A^\rhd$ for $A\in\cCat$ define a natural transformation
\begin{equation}\label{eq:i-j}
i\colon \id_{\cCat}\to(-)^\rhd\colon\cCat\to\cCat.
\end{equation}

If $A$ is a small category, then we can iterate the cocone construction and obtain the category $(A^\rhd)^\rhd$. This category is obtained from $A^\rhd$ by adding a new terminal object $\infty+1$. In particular, there is thus a morphism $\infty\to\infty+1$. (Similarly, in $(A^\lhd)^\lhd$ there is a morphism $-\infty-1\to -\infty$.) The category $(A^\rhd)^\rhd$ corepresents morphisms of cocones. In more detail, related to this category there are the following two functors.
\begin{enumerate}
\item The functor $s_A=i_{A^\rhd}\colon A^\rhd\to(A^\rhd)^\rhd$ is the component of the natural transformation $i$ at $A^\rhd$. Thus, the behavior of $s_A$ on objects is given by $a\mapsto a$ and $\infty\mapsto\infty$ and, given a morphism of cocones, restriction along $s_A$ yields the source of this morphism. 
\item In a similar way we also have the functor $t_A= i_A^\rhd\colon A^\rhd\to(A^\rhd)^\rhd$ obtained from $i_A\colon A\to A^\rhd$ by an application of the cocone functor. On objects the functor $t_A$ is given by $a\mapsto a$ and $\infty\mapsto \infty+1$ and, given a morphism of cocones, restriction along $t_A$ yields the target of this morphism.
\end{enumerate}
(An alternative description of the category $(A^\rhd)^\rhd$ is as the \emph{join} $A\ast[1]$ of $A$ and $[1]=(0<1)$. In that description, the above functors are induced by $0,1\colon\bbone\to[1]$, i.e., we have $s_A=\id_A\ast 0$ and $t_A=\id_A\ast 1$.)

If \D is a derivator, then we refer to $\D((A^\rhd)^\rhd)$ as the category of \textbf{morphisms of cocones} (with base~$A$). The category comes with \textbf{source} and \textbf{target functors} 
\[
s_A^\ast,t_A^\ast\colon\D((A^\rhd)^\rhd)\to\D(A^\rhd).
\] 
\end{con}

We now show that $(t_A)_!$ forms the intended comparison maps. Related to \autoref{con:mor-of-cones} there are the naturality squares
\begin{equation}\label{eq:cone-cone}
\vcenter{
\xymatrix{
A\ar[r]^-{i_A}\ar[d]_-{i_A}\drtwocell\omit{\id}&A^\rhd\ar[d]^-{t_A}&&
A\ar[r]^-{i_A}\ar[d]_-{i_A}&A^\lhd\ar[d]^-{t_A}\\
A^\rhd\ar[r]_-{s_A}&(A^\rhd)^\rhd,&&
A^\lhd\ar[r]_-{s_A}&(A^\lhd)^\lhd.\ultwocell\omit{\id}
}
}
\end{equation}

\begin{lem}[{\cite[Lem.~8.6]{gst:basic}}]\label{lem:he-cone-2}
For every $A\in\cCat$ the squares \eqref{eq:cone-cone} are homotopy exact.
\end{lem}
\begin{proof}
We take care of the square on the left. Since the square commutes and since the vertical functors are fully faithful, it suffices by \autoref{lem:he-ex-ff} to show that the canonical mate is an isomorphism at $\infty\in A^\rhd$. To reformulate this we consider the pasting on the left in 
\[
\xymatrix{
A\ar[r]^-\id\ar[d]&A\ar[r]^-{i_A}\ar[d]_-{i_A}&A^\rhd\ar[d]^-{t_A}&&A\ar[r]^-{i_A}\ar[d]&A^\rhd\ar[d]^-{t_A}\\
\bbone\ar[r]_-\infty&A^\rhd\ar[r]_-{s_A}&(A^\rhd)^\rhd,&&\bbone\ar[r]_-\infty&(A^\rhd)^\rhd.
}
\]
As the above two pastings agree, the compatibility of mates with pasting and the homotopy exactness of the square to the left (\autoref{prop:htpy-ex-cocone}) imply that it suffices to show that the square to the very right is homotopy exact. As this square is isomorphic to a slice square we are done by (Der4). 
\end{proof}

To conclude the proof of \autoref{lem:he-cone-2} it remains to establish the following lemma which is also of independent interest. In that lemma we consider a natural isomorphism living in a square of small categories 
\begin{equation}\label{eq:base}
\vcenter{
\xymatrix{
A'\ar[r]^-j\ar[d]_-{u'}\drtwocell\omit{\cong}&A\ar[d]^-u\\
B'\ar[r]_-k&B.
}
}
\end{equation}
The point of the following result is that it suffices to check objects in $B'\text{-}u'(A')$.

\begin{lem}\label{lem:he-ex-ff}
Let \eqref{eq:base} be a natural isomorphism in \cCat such that $u$ and $u'$ are fully faithful. The square \eqref{eq:base} is homotopy exact if and only if for all derivators the canonical mate $\beta_{b'}\colon (u'_! j^\ast)_{b'}\to (k^\ast u_!)_{b'}$ is an isomorphism for all $b'\in B'\text{-}u'(A')$.
\end{lem}
\begin{proof}
By axiom (Der2) of a derivator we have to show that, under the above assumptions, the canonical mate is always an isomorphism at objects of the form $u'(a'), a'\in A'$. To this end, we consider the pasting on the left in 
\[
\xymatrix@-0.5pc{
\bbone\ar[r]\ar[d]_-\id\drtwocell\omit{\id}&(u'/u'(a'))\ar[r]\ar[d]_-\pi\drtwocell\omit{}&A'\ar[r]^-j\ar[d]^-{u'}\drtwocell\omit{\cong}&A\ar[d]^-u&&
\bbone\ar[r]\ar[d]_-\id\drtwocell\omit{\id}&(u/uj(a'))\ar[r]\ar[d]_-\pi\drtwocell\omit{}&A\ar[d]^-u\\
\bbone\ar[r]_-\id&\bbone\ar[r]_-{u'(a')}&B'\ar[r]_-k&B,&&
\bbone\ar[r]_-\id&\bbone\ar[r]_-{uj(a')}&B,
}
\]
in which the square in the middle is a slice square and hence homotopy exact by axiom (Der4). The morphism $\bbone\to(u'/u'(a'))$ classifies the terminal object $(a',\id)$ (since $u'$ is fully faithful, this is a terminal object), and the square on the left is hence also homotopy exact (\cite[Prop.~1.18]{groth:ptstab}). The functoriality of mates with pasting implies that $\beta_{u'(a')}$ is an isomorphism if and only if the canonical mate associated to the pasting on the left is an isomorphism. 

In the pasting on the right, the square on the right is a slice square and the left square is induced by the functor $\bbone\to (u/uj(a'))$ classifying the terminal object $(j(a'),\id\colon uj(a')\to uj(a'))$ (using this time that $u$ is fully faithful). Up to a vertical pasting with the component of the isomorphism~\eqref{eq:base} at $a'$, the above two pasting agree. Similar arguments as above hence show that the pasting on the right and thus also the pasting on the left is homotopy exact, thereby concluding the proof.
\end{proof}

Thus, by \autoref{lem:he-cone-2} the source $s_A^\ast(t_A)_!(X)$ of the morphism of cocones $(t_A)_!(X)$ is a colimiting cocone, and this already characterizes the essential image of $(t_A)_!$.

\begin{prop}\label{prop:ess-im-target}
Let \D be a derivator and let $A\in\cCat$. The left Kan extension $(t_A)_!\colon\D(A^\rhd)\to\D((A^\rhd)^\rhd)$ is fully faithful and $Y\in\D((A^\rhd)^\rhd)$ lies in the essential image if and only if the source cocone $s_A^\ast Y\in\D(A^\rhd)$ is colimiting.
\end{prop}
\begin{proof}
Since $t_A$ is fully faithful, so is $(t_A)_!\colon\D(A^\rhd)\to\D((A^\rhd)^\rhd)$ and the essential image consists by \cite[Lem.~1.21]{groth:ptstab} precisely of those $Y$ such that the counit $\varepsilon\colon (t_A)_! t_A^\ast(Y)\to Y$ is an isomorphism at $\infty$. To reformulate this we consider the pasting on the left in
\begin{equation}\label{eq:prop-morph-cone}
\vcenter{
\xymatrix{
A\ar[r]^-{i_A}\ar[d]_-{\pi_A}\drtwocell\omit{}&A^\rhd\ar[r]^-{t_A}\ar[d]_-{t_A}\drtwocell\omit{\id}&(A^\rhd)^\rhd\ar[d]^-\id&
A\ar[r]^-\id\ar[d]_-{\pi_A}\drtwocell\omit{}&A\ar[r]^-{i_A}\ar[d]_-{i_A}\drtwocell\omit{\id}&A^\rhd\ar[r]^-{s_A}\ar[d]^-\id\drtwocell\omit{\id}&(A^\rhd)^\rhd\ar[d]^-\id\\
\bbone\ar[r]_-\infty&(A^\rhd)^\rhd\ar[r]_-\id&(A^\rhd)^\rhd,&
\bbone\ar[r]_\infty&A^\rhd\ar[r]_-\id&A^\rhd\ar[r]_-{s_A}&(A^\rhd)^\rhd.
}
}
\end{equation}
In this pasting the square on the left is isomorphic to a slice square and hence homotopy exact by (Der4). Thus, the functoriality of mates with pasting implies that $Y$ lies in the essential image of $(t_A)_!$ if and only if the canonical mate of this pasting is an isomorphism on $Y$. Since the above two pastings agree, similar arguments including \autoref{prop:htpy-ex-cocone} show that the canonical mate of the pasting on the right is an isomorphism if and only if the source cocone $s_A^\ast Y$ is colimiting.
\end{proof}

\begin{defn}\label{defn:cocone-comp-map}
Let \D be a derivator, $A\in\cCat$, and $X\in\D(A^\rhd)$. The \textbf{(cocone) comparison map} is the coherent morphism
\begin{equation}\label{eq:cocone-comp-map}
(t_A)_!(X)_\infty\to(t_A)_!(X)_{\infty+1}.
\end{equation}
\end{defn}

\begin{prop}[{\cite[Lem.~8.7]{gst:basic}}]\label{prop:cocone-comparison}
Let \D be a derivator and let $A$ be small category. A cocone $X\in\D(A^\rhd)$ is colimiting if and only if the comparison map~\eqref{eq:cocone-comp-map} is an isomorphism.
\end{prop}
\begin{proof}
Since the functors $i_A\colon A\to A^\rhd$ and $t_A\colon A^\rhd\to (A^\rhd)^\rhd$ are fully faithful, it suffices to note that the pastings
\[
\xymatrix{
A\ar[r]^-\id\ar[d]_-{\pi_A}&A\ar[d]_-{i_A}\ar[r]^-{i_A}&A^\rhd\ar[d]_-\id\ar[r]^-\id&A^\rhd\ar[d]^-{t_A}\ar@{}[rd]|{=}&A\ar[r]^-{i_A}\ar[d]_-{\pi_A}&A^\rhd\ar[d]^-{t_A}\\
\bbone\ar[r]_-\infty&A^\rhd\ar[r]_-\id&A^\rhd\ar[r]_-{t_A}&(A^\rhd)^\rhd&\bbone\ar@{}[]!R(0.5);[r]!L(0.5)|{\Downarrow}\ar@/^1.5ex/[]!R(0.5);[r]!L(0.5)^-\infty\ar@/_1.5ex/[]!R(0.5);[r]!L(0.5)_-{\infty+1}&(A^\rhd)^\rhd
}
\]
agree, to observe that with the exception of the second square from the left all squares in the above diagrams are homotopy exact and to apply the functoriality of mates with pasting.
\end{proof}

The proof of this proposition shows that for $X\in\D(A^\rhd)$ the canonical mate \eqref{eq:cone-II} and the comparison map \eqref{eq:cocone-comp-map} sit in a commutative square
\[
\xymatrix{
\colim_A i_A^\ast(X)\ar[r]\ar[d]_-\cong& X_\infty\ar[d]^-\cong\\
(t_A)_!(X)_\infty\ar[r]&(t_A)_!(X)_{\infty+1}.
}
\]

\section{Preservation of colimits}
\label{sec:ctns-mor}

Recall that a morphism of derivators $F\colon\D\to\E$ is a pseudo-natural transformation, thereby coming with structure isomorphisms. We can use these structure isomorphisms $\gamma_u^{-1}\colon F_A u^\ast\to u^\ast F_B$ and $\gamma_u\colon u^\ast F_B\to F_A u^\ast$,
\[
\xymatrix{
\D(A)\ar[r]^-{F_A}&\E(A)&&
\D(A)\ar[r]^-{F_A}&\E(A)\dltwocell\omit{\cong}\\
\D(B)\ar[r]_-{F_B}\ar[u]^-{u^\ast}\urtwocell\omit{\cong}&\E(B),\ar[u]_-{u^\ast}&&
\D(B)\ar[r]_-{F_B}\ar[u]^-{u^\ast}&\E(B),\ar[u]_-{u^\ast}
}
\]
in order to talk about morphisms of derivators which preserve Kan extensions. In fact, associated to these natural transformations there are the canonical mates
\begin{gather}
u_! F_A\stackrel{\eta}{\to} u_! F_A u^\ast u_! \stackrel{\gamma_u^{-1}}{\to} u_! u^\ast F_B u_!\stackrel{\varepsilon}{\to} F_B u_!, \label{eq:mate-mor-cocont}\\
F_B u_\ast \stackrel{\eta}{\to} u_\ast u^\ast F_B u_\ast \stackrel{\gamma_u}{\to} u_\ast F_A u^\ast u_\ast \stackrel{\varepsilon}{\to} u_\ast F_A.\label{eq:mate-mor-cont}
\end{gather}

\begin{defn}\label{defn:cocontinuous}
Let $F\colon\D\to\E$ be a morphism of derivators and let $u\colon A\to B$ be in $\cCat$.
\begin{enumerate}
\item The morphism $F$ \textbf{preserves left Kan extensions along~$u$} if the canonical mate \eqref{eq:mate-mor-cocont} is an isomorphism. The morphism is \textbf{cocontinuous} if it preserves left Kan extensions along all $u$. The morphism \textbf{preserves colimits of shape $A$} if it preserves left Kan extensions along $\pi_A\colon A\to\bbone$.
\item The morphism $F$ \textbf{preserves right Kan extensions along~$u$} if the canonical mate \eqref{eq:mate-mor-cont} is an isomorphism. The morphism is \textbf{continuous} if it preserves right Kan extensions along all $u$. The morphism \textbf{preserves limits of shape $A$} if it preserves right Kan extensions along $\pi_A\colon A\to\bbone$.
\end{enumerate}
\end{defn}

By duality, we allow ourselves to state and prove results only for morphisms which preserve certain colimits or left Kan extensions. The dual statements follow from the following \textbf{duality principle}.

\begin{lem}\label{lem:dual-prin-mor}
The following are equivalent for a morphism of derivators $F\colon\D\to\E$ and a functor $u\colon A\to B$ in $\cCat$. 
\begin{enumerate}
\item $F\colon\D\to\E$ preserves left Kan extensions along $u\colon A\to B$.
\item $F\op\colon\D\op\to\E\op$ preserves right Kan extensions along $u\op\colon A\op\to B\op$.
\end{enumerate}
\end{lem}
\begin{proof}
We leave it to the reader to unravel definitions in order to verify that the canonical mate \eqref{eq:mate-mor-cont} for $F\op$ and $(u\op)_\ast$ can be chosen to be the opposite of the canonical mate \eqref{eq:mate-mor-cocont} for $F$ and $u_!$. Since the passage to opposite natural transformations preserves and reflects natural isomorphisms this concludes the proof.
\end{proof}

\begin{lem}[{\cite[Prop.~2.3]{groth:ptstab}}]\label{lem:cocont-colimits}
A morphism of derivators is cocontinuous if and only if it preserves colimits.
\end{lem}

\begin{rmk}\label{rmk:cocont-colimits}
The proof of \cite[Prop.~2.3]{groth:ptstab} establishes the following more precise statement which will be useful later. Let $F\colon\D\to\E$ be a morphism of derivators and let $u\colon A\to B$ be a functor between small categories. If $F$ preserves colimits of shape $(u/b),b\in B,$ then $F$ preserves left Kan extensions along~$u$.
\end{rmk}

We defined a morphism of derivators to preserve colimits of shape~$A$ if it preserves left Kan extension along $\pi_A\colon A\to\bbone$. The next goal is to relate this to the preservation of colimiting cocones. To this end, we collect the following more general results concerning left Kan extensions along fully faithful functors which are also of independent interest.

\begin{lem}\label{lem:ff-kan-ess-im}
Let $F\colon\D\to\E$ be a morphism of derivators and let $u\colon A\to B$ be fully faithful. The morphism $F$ preserves left Kan extensions along $u$ if and only if $Fu_!(X),X\in\D(A),$ lies in the essential image of $u_!\colon\E(A)\to\E(B)$.
\end{lem}
\begin{proof}
If $F$ preserves left Kan extensions along $u$, then the mate $u_!F(X)\to Fu_!(X)$ is an isomorphism, so $Fu_!(X)$ lies in the essential image of $u_!\colon\E(A)\to\E(B)$. Let us conversely assume that $Fu_!(X)$ lies in the essential image of $u_!$. Since $u_!$ is fully faithful, this is the case if and only if the counit $\varepsilon\colon u_!u^\ast\to\id$ is an isomorphism on $Fu_!(X)$. But this implies that the canonical mate \eqref{eq:mate-mor-cocont} given by
\[
u_!F(X)\stackrel{\eta}{\to}u_!Fu^\ast u_!(X)\toiso u_!u^\ast Fu_!(X)\stackrel{\varepsilon}{\to} Fu_!(X)
\]
is an isomorphism. In fact, the second morphism is the pseudo-naturality isomorphism $\gamma_{F,u}^{-1}$ belonging to~$F$, $\eta$ is an isomorphism since $u_!$ is fully faithful, and $\varepsilon$ was just observed to be an isomorphism.
\end{proof}

\begin{lem}\label{lem:ff-kan-morphism}
Let $F\colon \D\to\E$ be a morphism of derivators, let $u\colon A\to B$ be fully faithful, and let $X\in\D(A)$. The canonical mate $\beta\colon u_!F(X)\to Fu_!(X)$ \eqref{eq:mate-mor-cocont} is an isomorphism in $\E(B)$ if and only if $\beta_b\colon (u_!F)(X)_b\to (Fu_!)(X)_b$ is an isomorphism for all $b\in B-u(A)$.
\end{lem}
\begin{proof}
Since isomorphisms in \E are detected pointwise by axiom (Der2), it is enough to show that $u^\ast\beta\colon u^\ast u_!F(X)\to u^\ast Fu_!(X)$ is always an isomorphism. This restricted canonical mate $u^\ast\beta$ is given by
\[
u^\ast u_!F(X)\stackrel{\eta}{\to}u^\ast u_!Fu^\ast u_!(X)\toiso u^\ast u_!u^\ast Fu_!(X)\stackrel{u^\ast\varepsilon}{\to} u^\ast Fu_!(X).
\]
Since $u$ is fully faithful, the unit $\eta\colon\id\to u^\ast u_!$ is an isomorphism. Finally, the triangular identity 
\[
\id=u^\ast\varepsilon\circ \eta u^\ast\colon u^\ast\to u^\ast u_! u^\ast\to u^\ast
\]
shows that also $u^\ast\varepsilon$ is an isomorphism, and $u^\ast\beta$ is hence always a natural isomorphism.
\end{proof}

The point here is that it suffices to control the objects in $B-u(A)$.

\begin{prop}\label{prop:mor-colim-vs-cocone}
Let $F\colon\D\to\E$ be a morphism of derivators and let $A\in\cCat.$ The following are equivalent.
\begin{enumerate}
\item The morphism $F$ preserves colimits of shape $A$.
\item The morphism $F$ preserves left Kan extensions along $i_A\colon A\to A^\rhd$.
\item The functor $F\colon\D(A^\rhd)\to\E(A^\rhd)$ preserves colimiting cocones.
\end{enumerate}
\end{prop}
\begin{proof}
We begin by showing that the first two statements are equivalent. The morphism $F$ preserves left Kan extensions along the fully faithful functor $i_A$ if and only if the canonical mate $(i_A)_!F\to F(i_A)_!$ is an isomorphism at $\infty\in A^\rhd$ (\autoref{lem:ff-kan-morphism}). To express this differently, let us consider the pasting on the left in
\begin{equation}
  \vcenter{\xymatrix@-0.4pc{
      \D(A) \ar[r]^-{F_A}&
      \E(A)\ar[r]^-{\id} &
      \E(A)\\
      \D(A^\rhd)\ar[r]_-{F_{A^\rhd}}\ar[u]^-{i_A^\ast} \urtwocell\omit{}&
      \E(A^\rhd)\ar[r]_-{\infty^\ast}\ar[u]^-{i_A^\ast} \urtwocell\omit{}&
      \E(\bbone),\ar[u]_-{\pi^\ast}
    }}
  \qquad
  \vcenter{\xymatrix@-0.4pc{
      \D(A) \ar[r]^-{\id}&
      \D(A)\ar[r]^-{F_A} &
      \E(A)\\
      \D(A^\rhd)\ar[r]_-{\infty^\ast}\ar[u]^-{i_A^\ast} \urtwocell\omit{}&
      \D(\bbone)\ar[r]_-{F_\bbone}\ar[u]^-{\pi^\ast} \urtwocell\omit{}&
      \E(\bbone),\ar[u]_-{\pi^\ast}
    }}
\end{equation}
The homotopy exactness of the square~\eqref{eq:htpy-exact-cocone} and the functoriality of mates with pasting, imply that $F$ preserves left Kan extensions along $i_A$ if and only if the canonical mate of the pasting on the left is an isomorphism. Since the above two pastings agree up to a vertical pasting by the pseudo-naturality isomorphism $\gamma_\infty$, we can again invoke \autoref{prop:htpy-ex-cocone} and the functoriality of mates with pasting to conclude that $F$ preserves colimits of shape~$A$ if and only if it preserves left Kan extensions along $i_A$. Finally, \autoref{lem:ff-kan-ess-im} establishes the equivalence of the second and the third statement.
\end{proof}

By the very definition, a morphism of derivators $F\colon\D\to\E$ preserves $A$-shaped colimits if a certain canonical mate is an isomorphism. This proposition allows us instead to simply verify that colimiting cocones are preserved. 

The following compatibility of the mates \eqref{eq:mate-mor-cocont} with adjunction (co)units will be useful later. Of course there is a dual statement concerning the mates \eqref{eq:mate-mor-cont}.

\begin{lem}\label{lem:units-vs-morphisms}
For a morphism of derivators $F\colon\D\to\E$ and $u\colon A\to B$ in $\cCat$ the following diagrams commute for every $X\in\D(A), Y\in\D(B)$,
\[
\xymatrix{
F(X)\ar[r]^-{F\eta}\ar[d]_-{\eta F}&Fu^\ast u_!(X)\ar[d]^-{\gamma^{-1}}&&
u_!Fu^\ast(Y)\ar[r]\ar[d]_-{\gamma^{-1}}&Fu_!u^\ast(Y)\ar[d]^-{F\varepsilon}\\
u^\ast u_! F(X)\ar[r]&u^\ast Fu_!(X),&&
u_!u^\ast F(Y)\ar[r]_-{\varepsilon F}&F(Y).
}
\]
\end{lem}
\begin{proof}
Plugging in the definition of the canonical mate $u_! F\to F u_!$ \eqref{eq:mate-mor-cocont}, in the first case it suffices to consider the following diagram,
\[
\xymatrix{
F\ar[r]^-{F\eta}\ar[d]_-{\eta F}&Fu^\ast u_!\ar[r]^-{\gamma^{-1}}\ar[d]_-{\eta F}&u^\ast F u_!\ar[d]_-{\eta u^\ast}\ar[rd]^-\id&\\
u^\ast u_! F\ar[r]_-{F\eta}&u^\ast u_! F u^\ast u_!\ar[r]_-{\gamma^{-1}}&u^\ast u_! u^\ast F u_!\ar[r]_-{u^\ast\varepsilon}&u^\ast F u_!.
}
\]
The two squares commute as naturality squares and the triangle commutes by a triangular identity. For the second claim, it is enough to consider the diagram
\[
\xymatrix{
u_! Fu^\ast\ar[r]^-{\eta u^\ast}\ar[dr]_-\id&u_!Fu^\ast u_! u^\ast\ar[d]^-{u^\ast\varepsilon}\ar[r]^-{\gamma^{-1}}&u_! u^\ast F u_! u^\ast\ar[d]_-{F\varepsilon}\ar[r]^-{\varepsilon F}&F u_! u^\ast\ar[d]^-{F\varepsilon}\\
&u_!Fu^\ast\ar[r]_-{\gamma^{-1}}&u_! u^\ast F\ar[r]_-{\varepsilon F}&F,
}
\]
which commutes for similar reasons.
\end{proof}

We conclude this section by the following compatibility of the mates \eqref{eq:mate-mor-cocont} and \eqref{eq:mate-mor-cont} with natural transformations.

\begin{lem}\label{lem:mates-cont-trafo}
Let $F,G\colon\D\to\E$ be morphisms of derivators, let $\alpha\colon F\to G$ be a natural transformation, and let $u\colon A\to B$. The following diagrams commute 
\[
\xymatrix{
u_! F_A\ar[r]\ar[d]_-{\alpha_A}&F_B u_!\ar[d]^-{\alpha_B}&&F_B u_\ast\ar[r]\ar[d]_-{\alpha_B}& u_\ast F_A\ar[d]^-{\alpha_A}\\
u_! G_A\ar[r]& G_B u_!,&&G_B u_\ast\ar[r]& u_\ast G_A.
}
\]
\end{lem}
\begin{proof}
By duality it suffices to take care of the first statement, and unraveling definitions this amounts to showing that the diagram
\[
\xymatrix{
u_!F_A\ar[r]^-\eta\ar[d]_-{\alpha_A}&u_!F_Au^\ast u_!\ar[r]^{\gamma^{-1}}\ar[d]_-{\alpha_A}&u_!u^\ast F_B u_!\ar[d]^-{\alpha_B}\ar[r]^-\varepsilon&F_B u_!\ar[d]^-{\alpha_B}\\
u_!G_A\ar[r]_-\eta&u_!G_Au^\ast u_!\ar[r]_-{\gamma^{-1}}&u_!u^\ast G_B u_!\ar[r]_-\varepsilon&G_B u_!
}
\]
commutes. Here, the outer two squares commute as naturality squares, while the remaining one commutes by the coherence properties of a modification.
\end{proof}

\section{Commutativity of limits and colimits}

In this short section we consider colimit and limit morphisms of derivators and their compatibility with limits and colimits. This leads to the question if limits and colimits in unrelated variables commute.

\begin{lem}\label{lem:lkan-commute-naive}
Let \D be a derivator, let $u\colon A\to B, v\colon B\to C$ be in $\cCat$, and let $X\in\D(A)$. There are canonical isomorphisms
\[
(v\circ u)_!(X)\cong v_!(u_!(X))\qquad\text{and}\qquad (\id_A)_!(X)\cong X.
\]
\end{lem}
\begin{proof}
This is immediate from the uniqueness of left adjoints and the relation
\[
(v\circ u)^\ast=u^\ast v^\ast\colon\D(C)\to\D(A).
\]
In the same way we obtain a canonical isomorphism $(\id_A)_!\cong\id_{\D(A)}.$
\end{proof}

There are the following immediate consequences.

\begin{cor}\label{cor-lkan-colim-comm}
Let \D be a derivator, let $u\colon A\to B$ in $\cCat$, and let $X\in\D(A)$. There are canonical isomorphisms
\[
\colim_A X\cong\colim_B u_!(X).
\]
\end{cor}
\begin{proof}
We simply apply \autoref{lem:lkan-commute-naive} to $\pi_A=\pi_B\circ u\colon A\to\bbone$.
\end{proof}

A variant is the following \textbf{Fubini theorem}, saying that left Kan extensions in unrelated variables commute.

\begin{cor}\label{cor:fubini}
Let \D be a derivator, let $u\colon A\to A'$ and $v\colon B\to B'$, and let $X\in\D(A\times B).$ There are canonical isomorphisms
\[
(u\times\id)_!(\id\times v)_!X\cong (u\times v)_! X\cong (\id\times v)_!(u\times \id)_! X.
\]
\end{cor}
\begin{proof}
Considering the naturality square
\[
\xymatrix{
A\times B\ar[r]^-{u\times\id}\ar[d]_-{\id\times v}&A'\times B\ar[d]^-{\id\times v}\\
A\times B'\ar[r]_-{u\times\id}&A'\times B',
}
\]
this is immediate from \autoref{lem:lkan-commute-naive}. Alternatively, in order to obtain a canonical isomorphism between the two outer expressions it suffices to note that the Kan extension morphism $u_!\colon\D^A\to\D^{A'}$ is a left adjoint and hence cocontinuous.
\end{proof}

Using suggestive notation, the Fubini theorem specializes as follows.

\begin{cor}
Let \D be a derivator, let $A,B\in\cCat,$ and let $X\in\D(A\times B)$. There are canonical isomorphisms
\[
\colim_A\colim_B X\cong\colim_{A\times B} X\cong\colim_B\colim_A X.
\]
\end{cor}
\begin{proof}
This is a special case of \autoref{cor:fubini}.
\end{proof}

Thus, colimits in unrelated variables commute in every derivator. A classical reference for such a result can already be found in \cite{vogt:commuting}. The results of this section of course dualize to yield similar statements for right Kan extensions. 

The mixed situation, i.e., the question whether limits and colimits in unrelated variables commute, is more subtle. Given two small categories $A$ and $B$, a derivator \D and $X\in\D(A\times B)$, we consider the canonical mate
\begin{align}\label{eq:colim-lim-comm}
(\pi_A)_!(\pi_B)_\ast X&\to (\pi_A)_!(\pi_B)_\ast (\pi_A)^\ast(\pi_A)_! X\\
&\toiso (\pi_A)_!(\pi_A)^\ast(\pi_B)_\ast (\pi_A)_! X\\
& \to (\pi_B)_\ast (\pi_A)_! X
\end{align}

\begin{defn}\label{defn:colim-lim-comm}
Let \D be a derivator and let $A,B\in\cCat$. We say that \textbf{colimits of shape $A$ and limits of shape $B$ commute} in \D if for every $X\in\D(A\times B)$ the canonical mate $\colim_A\mathrm{lim}_BX\to\mathrm{lim}_B\colim_A X$ \eqref{eq:colim-lim-comm} is an isomorphism.  
\end{defn}

In a similar way, given functors $u\colon A\to A'$ and $v\colon B\to B'$, we define that \textbf{left Kan extension along $u$ and right Kan extension along $v$ commute} by asking that the canonical mate
\begin{align}\label{eq:lkan-rkan-comm}
(u\times\id)_!(\id\times v)_\ast&\to (u\times\id)_!(\id\times v)_\ast (u\times\id)^\ast(u\times\id)_!\\
&\toiso (u\times\id)_!(u\times\id)^\ast(\id\times v)_\ast (u\times\id)_!\\
& \to (\id\times v)_\ast (u\times\id)_!
\end{align}
is an isomorphism. For less cumbersome terminology, we also say that $u_!$ and $v_\ast$ commute in \D.

To relate this to \S\ref{sec:ctns-mor} we make the following trivial observation.

\begin{lem}\label{lem:lkan-rkan-comm}
Let \D be a derivator and let $u\colon A\to A',v\colon B\to B'$ be in $\cCat$. The following are equivalent.
\begin{enumerate}
\item $u_!$ and $v_\ast$ commute in \D.
\item The morphism $v_\ast\colon\D^B\to\D^{B'}$ preserves left Kan extensions along $u$.
\item The morphism $u_!\colon\D^A\to\D^{A'}$ preserves right Kan extensions along $v$.
\end{enumerate}
\end{lem}
\begin{proof}
Unraveling definitions, we see that the canonical mate expressing that $u_!$ preserves right Kan extensions along $v$ is precisely the canonical mate \eqref{eq:lkan-rkan-comm}. In the case of the morphism $v_\ast$ it suffices to conjugate with restrictions along the symmetry constraints in $(\cCat,\times,\bbone)$.
\end{proof}

\begin{lem}\label{lem:comm-cone}
The following are equivalent for a derivator \D and $A,B\in\cCat$.
\begin{enumerate}
\item Colimits of shape $A$ and limits of shape $B$ commute in~\D.
\item Left Kan extensions along $i_A\colon A\to A^\rhd$ and right Kan extensions along $i_B\colon B\to B^\lhd$ commute in~\D.
\end{enumerate}
\end{lem}
\begin{proof}
This is immediate from \autoref{lem:lkan-rkan-comm} and \autoref{prop:mor-colim-vs-cocone}.
\end{proof}

In general, left and right Kan extensions in unrelated variables do not commute (as one observes by noting that this notion reduces to the usual one in represented derivators). As an additional illustration, in the sequel \cite{groth:char} we \emph{characterize} pointed and stable derivators, respectively, by the commutativity of certain left and right Kan extensions.

\section{Some closure and invariance properties}
\label{sec:ctns-closure}

In this section we collect some closure and invariance properties of morphisms of derivators preserving certain (co)limits or Kan extensions (see \cite{albert-kelly:closure} for a reference in the context of ordinary category theory). We again focus on left Kan extensions, and corresponding results for right Kan extensions follow by duality.

\begin{defn}\label{defn:eqiv-mor}
Two morphisms of derivators $F_1\colon\D_1\to\E_1$ and $F_2\colon\D_2\to\E_2$ are \textbf{equivalent}, if there are equivalences of derivators $\varphi\colon\D_1\simeq\D_2, \psi\colon\E_1\simeq\E_2$ and a natural isomorphism $\alpha\colon\psi\circ F_1\cong F_2\circ\varphi$,
\[
\xymatrix{
\D_1\ar[r]^-{F_1}\ar[d]_-\varphi^-\simeq\ar@{}[dr]|{\cong}&\E_1\ar[d]^-\psi_-\simeq\\
\D_2\ar[r]_-{F_2}&\E_2.
}
\]
The triple $(\varphi,\psi,\alpha)$ is an \textbf{equivalence} $F_1\to F_2$.
\end{defn}

\begin{prop}\label{prop:sorite-lim-I}
~ 
\begin{enumerate}
\item Equivalences and left adjoint morphisms of derivators are cocontinuous.
\item If $F$ and $G$ preserve left Kan extensions along $u$, then so does $G\circ F$.
\item If $F_1,F_2\colon\D\to\E$ are naturally isomorphic, then $F_1$ preserves left Kan extensions along $u$ if and only if $F_2$ does.
\item If $F_1\colon\D_1\to\E_1$ and $F_2\colon\D_2\to\E_2$ are equivalent, then $F_1$ preserves left Kan extensions along $u$ if and only if $F_2$ does.
\end{enumerate}
\end{prop}
\begin{proof}
The first statement is \cite[Prop.~2.9]{groth:ptstab} while the second and the third statements are \cite[Prop.~2.4]{groth:ptstab}. The fourth statement follows immediately from the first and the third one.
\end{proof} 

This proposition has a variant if we fix a morphism of derivators and let the functors in $\cCat$ vary. As a preparation we make the following construction.

\begin{con}\label{con:total-mate-kan}
Let \D be a derivator, let $u,v\colon A\to B$ be functors, and let $\alpha\colon u\to v$ be a natural transformation. The restriction functors $u^\ast,v^\ast\colon\D(B)\to\D(A)$ are related by the natural transformation $\alpha^\ast\colon u^\ast\to v^\ast.$ Since $u^\ast,v^\ast$ both admit left adjoints and right adjoints, we can hence consider the associated total mates or conjugate transformations
\begin{gather}
\alpha_!\colon v_!\stackrel{\eta}{\to} v_!u^\ast u_!\stackrel{\alpha^\ast}{\to} v_!v^\ast u_!\stackrel{\varepsilon}{\to} u_!, \label{eq:total-mate-lkan}\\
\alpha_\ast\colon v_\ast\stackrel{\eta}{\to} u_\ast u^\ast v_\ast\stackrel{\alpha^\ast}{\to} u_\ast v^\ast v_\ast\stackrel{\varepsilon}{\to} u_\ast. \label{eq:total-mate-rkan}
\end{gather}
\end{con}

The natural transformations \eqref{eq:total-mate-lkan} and \eqref{eq:total-mate-rkan} are compatible with the canonical morphisms \eqref{eq:mate-mor-cocont} and \eqref{eq:mate-mor-cont}, respectively. 

\begin{lem}\label{lem:total-mate-kan}
Let $F\colon\D\to\E$ be a morphism of derivator, let $u,v\colon A\to B$, and let $\alpha\colon u\to v$ be a natural transformation. The morphisms \eqref{eq:mate-mor-cocont} and \eqref{eq:total-mate-lkan} are compatible in that the following diagram commutes,
\[
\xymatrix{
u_! F_A\ar[r]&F_B u_!\\
v_! F_A\ar[u]^-{\alpha_!}\ar[r]& F_B v_!.\ar[u]_-{\alpha_!}
}
\]
\end{lem}
\begin{proof}
Let us consider the following two pastings which agree by the coherence properties of pseudo-natural transformations,
\[
\xymatrix{
\D(A)\ar[r]^-{F_A}&\E(A)\ar[r]^-\id&\E(A)&
\D(A)\ar[r]^-\id&\D(A)\ar[r]^-{F_A}&\E(A)\\
\D(B)\ar[r]_-{F_B}\ar[u]^-{u^\ast}\urtwocell\omit{}&\E(B)\ar[r]_-\id\ar[u]^-{u^\ast}\urtwocell\omit{}&\E(B),\ar[u]_-{v^\ast}&
\D(B)\ar[r]_-\id\ar[u]^-{u^\ast}\urtwocell\omit{}&\D(B)\ar[r]_-{F_B}\ar[u]^-{v^\ast}\urtwocell\omit{}&\E(B).\ar[u]_-{v^\ast}
}
\]
The functoriality of canonical mates with pasting and the description of the natural transformation $\alpha_!$ as the canonical mate associated to the inner two squares concludes the proof.
\end{proof}

The definition of equivalent morphisms in the $2$-category $\cDER$ of derivators (\autoref{defn:eqiv-mor}) has a variant in every $2$-category, hence there is also the notion of equivalent functors in $\cCat$.

\begin{prop}\label{prop:sorite-lim-II}
Let $F\colon\D\to\E$ be a morphism of derivators.
\begin{enumerate}
\item Every morphism of derivators preserves left Kan extensions along equivalences and left adjoint functors.
\item If $F$ preserves left Kan extensions along $u\colon A\to B$ and $v\colon B\to C$, then $F$ also preserves left Kan extensions along $vu\colon A\to C$.
\item If $u,v\colon A\to B$ are naturally isomorphic, then $F$ preserves left Kan extensions along $u$ if and only if $F$ preserves left Kan extensions along $v$.
\item If $u,v$ are equivalent functors in $\cCat$, then $F$ preserves left Kan extensions along $u$ if and only if $F$ preserves left Kan extensions along~$v$.
\end{enumerate}
\end{prop}
\begin{proof}
If $u\colon A\to B$ is an equivalence, then $u^\ast$ is part of an adjoint equivalence $(u_!,u^\ast)$. Hence the canonical mate $u_! F_A\to F_Bu_!$ factors as a composition of three isomorphisms,
\[
u_!F\toiso u_!Fu^\ast u_!\toiso u_!u^\ast Fu_!\toiso Fu_!,
\]
establishing the first part of the first statement. Let $(u,v)\colon A\rightleftarrows B$ be an adjunction. To conclude that every morphism preserves left Kan extensions along $u$ it suffices by \autoref{rmk:cocont-colimits} to show that every morphism preserves colimits of shape $(u/b),b\in B$. Since $(vb,\varepsilon_b\colon uvb\to b)\in(u/b)$ is a terminal object, this defines a homotopy final functor $\bbone\to(u/b)$ and the statement hence follows from the following lemma (\autoref{lem:final}). The second statement is immediate from the functoriality of mates with pasting. As for the third statement, if $\alpha\colon u\toiso v$ is a natural isomorphism, then so is $\alpha^\ast\colon u^\ast\toiso v^\ast$ and also the total mate $\alpha_!\colon v_!\to u_!$ \eqref{eq:total-mate-lkan}. The third statement is thus immediate from \autoref{lem:total-mate-kan}, and together with the first statement this implies statement four. 
\end{proof}

\begin{lem}\label{lem:final}
Let $F\colon\D\to\E$ be a morphism of derivators and let $u\colon A\to B$ be homotopy final. If $F$ preserves colimits of shape~$A$, then $F$ preserves colimits of shape~$B$.
\end{lem}
\begin{proof}
Let us recall that a functor $u\colon A\to B$ is homotopy final if and only if the square
\begin{equation}\label{eq:htpy-exact-final}
\vcenter{
\xymatrix{
A\ar[r]^-u\ar[d]_-{\pi_A}\xtwocell[1,1]{}\omit&B\ar[d]^-{\pi_B}\\
\bbone\ar[r]&\bbone
}
}
\end{equation}
is homotopy exact. To reformulate that~$F$ preserves colimits of shape~$B$ let us consider the pasting on the left in
\begin{equation}\label{eq:lem-final}
\vcenter{
\xymatrix{
\E(A)\ar[d]_-{(\pi_A)_!}& \E(B)\ar[l]_-{u^\ast}\ar[d]_-{(\pi_B)_!}& \D(B)\ar[d]^-{(\pi_B)_!}\ar[l]_-F&
\E(A)\ar[d]_-{(\pi_A)_!}& \D(A)\ar[l]_-F\ar[d]_-{(\pi_A)_!}& \D(B)\ar[d]^-{(\pi_B)_!}\ar[l]_-{u^\ast}\\
\E(\bbone)\urtwocell\omit{}&\E(\bbone)\ar[l]^-=\urtwocell\omit{}&\D(\bbone),\ar[l]^-F&
\E(\bbone)\urtwocell\omit{}&\D(\bbone)\ar[l]^-F\urtwocell\omit{}&\D(\bbone).\ar[l]^-=
}
}
\end{equation}
Using the homotopy exactness of~\eqref{eq:htpy-exact-final} and the functoriality of mates with pasting, we see that $F$ preserves colimits of shape~$B$ if and only if the pasting on the left is an isomorphism. Up to a vertical pasting by a pseudo-naturality isomorphism of $F$, this pasting agrees with the pasting on the right. Using again the homotopy exactness of \eqref{eq:htpy-exact-final} and our assumption on~$F$, we see that the pasting on the right indeed is an isomorphism, concluding the proof.
\end{proof} 

\begin{eg}
If $A\in\cCat$ admits a final object~$\omega$, then every morphism of derivators preserves colimits of shape $A$. In fact, in this case there is an adjunction $(\pi_A,\omega)\colon A\rightleftarrows\bbone$, and the result hence follows from \autoref{prop:sorite-lim-II} applied to $\pi_A$. Alternatively, the result follows from an application of \autoref{lem:final} to $\omega\colon\bbone\to A$.
\end{eg}

\begin{warn}
Note that, in general, the converse to \autoref{lem:final} is not true. Let $A\in\cCat$ and $\pi_A\colon A\to \bbone$ be the unique functor. The functor $\pi_A$ is homotopy final as soon as the nerve $NA$ is weakly contractible (\cite[Cor.~3.13]{gps:mayer}). Considering $A=\ulcorner$ for example in the represented case, this shows that, in general, the converse to \autoref{lem:final} fails.
\end{warn}

In order to also obtain a positive statement in the converse direction, it suffices to insist on $u^\ast$ being essentially surjective. 

\begin{lem}\label{lem:final-II}
Let $u\colon A\to B$ be homotopy final such that $u^\ast\colon\D(B)\to\D(A)$ is essentially surjective for every derivator \D. A morphism of derivators preserves colimits of shape~$A$ if and only if it preserves colimits of shape~$B$.
\end{lem}
\begin{proof}
By \autoref{lem:final} it suffices to show that a morphism $F\colon\D\to\E$ of derivators which preserves colimits of shape $B$ also preserves colimits of shape $A$. We again consider the pastings \eqref{eq:lem-final} which match up to a vertical pasting by a pseudo-naturality isomorphism. By assumption on $F$ and $u$, the pasting on the left is an isomorphism, hence so is the pasting on the right which is given by
\[
\colim_A F u^\ast\to F \colim_A u^\ast \toiso F \colim_B.
\]
Since the restriction morphism $u^\ast$ is essentially surjective, we conclude that $F$ preserves colimits of shape $A$.
\end{proof} 

\begin{cor}\label{cor:final-III}
Let $(l,r)\colon B\rightleftarrows A$ be a reflective localization, i.e., an adjunction such that $r$ is fully faithful. A morphism of derivators preserves colimits of shape $A$ if and only if it preserves colimits of shape $B$.
\end{cor}
\begin{proof}
As a right adjoint the functor $r\colon A\to B$ is homotopy final, hence by \autoref{lem:final-II} it remains to show that $r^\ast\colon\D(B)\to\D(A)$ is essentially surjective for every derivator~\D. But if $(l,r,\eta\colon\id\to rl,\varepsilon\colon lr\to\id)$ is an adjunction, then we obtain an induced adjunction $(r^\ast,l^\ast,\eta^\ast\colon\id\to l^\ast r^\ast,\varepsilon^\ast\colon r^\ast l^\ast\to\id)$. Since $r$ is fully faithful, $\varepsilon\colon lr\to \id$ is an isomorphism, hence so is $\varepsilon^\ast\colon r^\ast l^\ast\to\id$, and this implies that $r^\ast$ is essentially surjective.
\end{proof}

We collect an additional closure property of the class of functors along which a fixed morphism of derivators preserves left Kan extensions. This cancellation property will be useful in \S\ref{sec:exact-htpy-finite}.

\begin{lem}\label{lem:cancel-lkan}
Let $F\colon\D\to\E$ be a morphism of derivators and let $u\colon A\to B, v\colon B\to C$ be fully faithful functors in $\cCat$. If $F$ preserves left Kan extensions along $vu\colon A\to C$, then $F$ preserves left Kan extensions along $u$.
\end{lem}
\begin{proof}
Since $u$ is fully faithful, it suffices by \autoref{lem:ff-kan-ess-im} to show that $Fu_!(X),$ $X\in\D(A),$ lies in the essential image of $u_!\colon\E(A)\to\E(B)$. There is the following chain of natural isomorphisms establishing this fact,
\begin{align}
F(u_!X)&\cong F(v^\ast v_!u_!X)\\
&\cong v^\ast F(v_!u_!X)\\
&\cong v^\ast F((vu)_!X)\\
&\cong v^\ast(vu)_! F(X)\\
&\cong v^\ast v_!u_! F(X)\\
&\cong u_!F(X).
\end{align}
In fact, these isomorphisms are respectively given by the fully faithfulness of $v_!$, the pseudo-naturality of $F$, the uniqueness of adjoints (\autoref{lem:lkan-commute-naive}), the assumption on $F$, and again the uniqueness of adjoints and the fully faithfulness of $v_!$. 
\end{proof}

\begin{warn}\label{warn:cancel-lkan}
The dual version of \autoref{lem:cancel-lkan} allows us to conclude something about right Kan extensions along the first functor. For a counterexample to the statement making conclusion about the second functor see \autoref{warn:cancel-lkan-2}.
\end{warn}

\section{Continuity and parameters}
\label{sec:paras}

In this section we discuss some closure properties related to the passage to parametrized versions of morphisms and natural transformations. This allows us to reformulate questions related to cocontinuity internally to the $2$-category of derivators.

\begin{prop}[{\cite[Prop.~2.5]{groth:ptstab}}]\label{prop:restriction-continuous}
For every derivator \D and every functor $v\colon B\to B'$ the restriction morphism $v^\ast\colon\D^{B'}\to \D^{B}$ is continuous and cocontinuous.
\end{prop}

\begin{prop}[{\cite[Cor.~2.7]{groth:ptstab}.}]\label{prop:shift-cocont}
Let $F\colon\D\to\E$ be a morphism of derivators, let $u\colon A\to A'$, and let $B\in\cCat$. If $F$ preserves left Kan extensions along $u$, then $F$ also preserves left Kan extensions along $\id\times u\colon B\times A\to B\times A'$. 
\end{prop}

\begin{cor}\label{cor:shift-cocont}
A morphism of derivators $F\colon\D\to\E$ preserves left Kan extensions along $u\colon A\to A'$ if and only if $F^B\colon\D^B\to\E^B,B\in\cCat$, preserves left Kan extensions along $u$.
\end{cor}

There is the following closely related result.

\begin{prop}\label{prop:cocont-ptws}
Let $\D,\E$ be derivators, let $u\colon A\to A'$, let $B\in\cCat,$ and let $F\colon\D\to\E^B$ be a morphism of derivators. The morphism $F$ preserves left Kan extensions along $u$ if and only if $b^\ast F\colon\D\to\E,b\in B,$ preserves left Kan extensions along $u$.
\end{prop}
\begin{proof}
If $F$ preserves left Kan extensions along $u$, then so does $b^\ast F$ since evaluation morphisms are cocontinuous (\autoref{prop:restriction-continuous} and \autoref{prop:sorite-lim-I}). Conversely, in order to conclude that $F$ preserves left Kan extensions along $u$ we have to show that the canonical mate associated to the left square in 
\[
\xymatrix{
\D(A)\ar[r]^-F&\E(B\times A)\ar[r]^-{b^\ast}&\E(A)\\
\D(A')\ar[r]_-F\ar[u]^-{u^\ast}\urtwocell\omit{}&\E(B\times A')\ar[r]_-{b^\ast}\ar[u]^-{(\id\times u)^\ast}\urtwocell\omit{}&\E(A'),\ar[u]_-{u^\ast}
}
\]
is an isomorphism. By (Der2) it suffices to check this at every object $b\in B$, which, by the cocontinuity of $b^\ast$ (\autoref{prop:restriction-continuous}), is the case as soon as the canonical mate of the above pasting is an isomorphism for every $b\in B$. Since this precisely means that $b^\ast F,b\in B,$ preserves left Kan extensions along $u$, this concludes the proof.
\end{proof}

Based on these compatibilities, one can show that most of the results obtained in this paper have parametrized reformulations internally to the $2$-category of derivators. We begin by recalling that the passage to shifted derivators defines a pseudo-functor of two variables
\[
(-)^{(-)}\colon\cCat\op\times\cDER\to\cDER\colon (A,\D)\mapsto\D^A.
\]
While the partial pseudo-functors $(-)^A\colon\cDER\to\cDER$ and $\D^{(-)}\colon\cCat\op\to\cDER$ are actual $2$-functors, given a morphism of derivators $F\colon\D\to\E$ and a functor $u\colon A\to B$, the following diagram commutes up to the invertible natural transformation $\gamma_u\colon u^\ast F^B\toiso F^A u^\ast$,
\begin{equation}\label{eq:pseudo-para}
\vcenter{
\xymatrix{
\D^B\ar[r]^-{F^B}\ar[d]_-{u^\ast}\drtwocell\omit{\cong}&\E^B\ar[d]^-{u^\ast}\\
\D^A\ar[r]_-{F^A}&\E^A,
}
}
\end{equation}
given by the following lemma.

\begin{lem}
Let $F\colon\D\to\E$ be a morphism of prederivators and let $u\colon A\to B$ be a functor. The pseudo-naturality constraints $\gamma_{u\times\id_C},C\in\cCat,$ of $F$ assemble to an invertible modification $\gamma_u\colon u^\ast F^B\to F^Au^\ast$ populating \eqref{eq:pseudo-para}.
\end{lem}
\begin{proof}
This follows from a direct verification.
\end{proof}

Considering Kan extensions instead of restrictions we obtain the following.

\begin{lem}\label{lem:lax-lkan-prep}
Let $F\colon\D\to\E$ be a morphism of derivators and let $u\colon A\to B$ be a functor. The canonical mates \eqref{eq:mate-mor-cocont} for $u\times\id_C\colon A\times C\to  B\times C,C\in\cCat$, define a natural transformation $u_! F^A\to F^B u_!$,
\begin{equation}\label{eq:lkan-oplax}
\vcenter{
\xymatrix{
\D^A\ar[r]^-{F^A}\ar[d]_-{u_!}\drtwocell\omit&\E^A\ar[d]^-{u_!}\\
\D^B\ar[r]_-{F^B}&\E^B.
}
}
\end{equation}
\end{lem}
\begin{proof}
Since there are an invertible modification \eqref{eq:pseudo-para} and adjunctions of derivators $(u_!,u^\ast)$ for \D and for \E, we can consider the natural transformation
\[
u_! F^A\stackrel{\eta}{\to} u_! F^A u^\ast u_!\stackrel{\gamma_u^{-1}}{\to}u_!u^\ast F_B u_!\stackrel{\varepsilon}{\to} F_B u_!
\]
of derivators. This is an instance of the calculus of mates internally to the $2$-category $\cDER$. Unraveling definitions, this mate has as components the canonical mates \eqref{eq:mate-mor-cocont} associated to $u\times\id_C,C\in\cCat$.
\end{proof}

\begin{lem}\label{lem:lax-lkan}
Let $u\colon A\to B$ be a functor between small categories. There is a lax natural transformation
\[
u_!\colon(-)^A\to(-)^B\colon\cDER\to\cDER
\]
given by the morphisms $u_!\colon\D^A\to\D^B,\D\in\cDER,$ and the natural transformations $\gamma_u\colon u_! F^A\to F^B u_!$ \eqref{eq:lkan-oplax} for all $F\colon\D\to\E$ in $\cDER$.
\end{lem}
\begin{proof}
There are three coherence properties to be verified for such a lax natural transformation. The first two ask for a compatibility with respect to composition of morphisms of derivators and with respect to identity morphisms. These two are immediate from the functoriality of mates with pasting. Similarly, as a consequence of \autoref{lem:mates-cont-trafo} and \autoref{lem:lax-lkan-prep}, if $F,G\colon\D\to\E$ are morphisms of derivators and if $\alpha\colon F\to G$ is a natural transformation, then for every $u\colon A\to B$ there is the pasting relation
\[
\xymatrix{
u_!F^A\ar[r]\ar[d]_-\alpha&F^Bu_!\ar[d]^-\alpha&&\D^A\ar[r]^{F}\ar[d]_-{u_!}\drtwocell\omit{}&\E^A\ar[d]^-{u_!}\ar@{}[dr]|{=}&\D^A\ar@{}[]!R(0.5);[r]!L(0.5)|{\Downarrow_{\alpha}}\ar@/^1.5ex/[]!R(0.5);[r]!L(0.5)^{F}\ar@/_1.5ex/[]!R(0.5);[r]!L(0.5)_{G}\ar[d]_-{u_!}\drtwocell\omit{}&\E^A\ar[d]^-{u_!}\\
u_!G^A\ar[r]&G^Bu_!,&&\D^B\ar@{}[]!R(0.5);[r]!L(0.5)|{\Downarrow_{\alpha}}\ar@/^1.5ex/[]!R(0.5);[r]!L(0.5)^{F}\ar@/_1.5ex/[]!R(0.5);[r]!L(0.5)_{G}&\E^B&\D^B\ar[r]_-{G}&\E^B,
}
\]
thereby establishing the remaining coherence property.
\end{proof}

\begin{lem}\label{lem:nat-iso-under}
Let $F,G\colon\D\to\E$ be morphisms of derivators. A natural transformation $\alpha\colon F\to G$ is a natural isomorphism if and only if the underlying natural transformation $\alpha_\bbone\colon F_\bbone\to G_\bbone$ is invertible. 
\end{lem}
\begin{proof}
It is easy to check that a natural transformation $\alpha$ of derivators is invertible if and only if all components $\alpha_A,A\in\cCat,$ are invertible, and it remains to show that if $\alpha_\bbone$ is invertible then every $\alpha_A,A\in\cCat,$ is invertible. Since isomorphisms in $\E(A)$ are detected pointwise, it suffices to show that $a^\ast\alpha_A,a\in A,$ is an isomorphism in $\E(\bbone)$. Associated to $a\colon\bbone\to A$ there is the pasting relation
\[
\xymatrixcolsep{3pc}\xymatrixrowsep{3pc}
\xymatrix{
\D(\bbone)\rtwocell_{F_\bbone}^{G_\bbone}{^\alpha_\bbone\;\;}&\E(\bbone)\ar@{}[rd]|{=}&
\D(\bbone)\ar[r]^-{G_\bbone}&\E(\bbone)\\
\D(A)\ar[r]_-{F_A}\ar[u]^-{a^\ast}&\E(A)\ar[u]_-{a^\ast}\ultwocell\omit{\gamma}&
\D(A)\ar[u]^-{a^\ast}\rtwocell_{F_A}^{G_A}{^\alpha_A\;\;}&\E(A),\ar[u]_-{a^\ast}\ultwocell\omit{\gamma}
}
\]
expressing one of the coherence properties of a modification. Since $\gamma_{F,a},\gamma_{G,a},$ and $\alpha_\bbone$ are invertible, the same is true for $a^\ast\alpha_A$.
\end{proof}

\begin{rmk}
\autoref{lem:nat-iso-under} shows that the lax natural transformation constructed in \autoref{lem:lax-lkan} restricts to a pseudo-natural isomorphism on the sub-$2$-category given by all derivators, the morphisms which preserve left Kan extensions along $u$, and all natural transformations of derivators. Alternatively, this also follows from \autoref{prop:shift-cocont}.
\end{rmk}

This remark is of particular interest in the context of stable derivators and exact morphisms; see \S\ref{sec:exact-htpy-finite}.

\section{Coproduct preserving morphisms}
\label{sec:coproducts}

In this section we revisit the existence of (co)products in derivators \cite[Prop.~1.7]{groth:ptstab} and show that homotopy (co)products and categorical (co)products agree in a certain precise sense. Similarly, a morphism of derivators preserves homotopy (co)products if and only if it preserves categorical (co)products.

\begin{prop}\label{prop:coprod-match}
Let \D be a derivator and let $S$ be a discrete category. A cocone $X\in\D(S^\rhd)$ is a coproduct cocone if and only if $\ndia_{S^\rhd}(X)\colon S^\rhd\to \D(\bbone)$ is a coproduct cocone, i.e., it exhibits $X_\infty$ as the coproduct of $X_s\in\D(\bbone),s\in S$. 
\end{prop}
\begin{proof}
Let us consider the pasting diagram \autoref{fig:coprod} in which the canonical isomorphism in the top triangle follows from the construction of coproducts in $\D(\bbone)$. The bottom triangle is the adjunction counit, the top square commutes by the strictness of underlying diagram morphisms, and the bottom square is populated by the natural transformation induced by \eqref{eq:htpy-exact-cocone} in the special case of $A=S$. Note that $\ndia_S\pi_S^\ast=\Delta_S$ and the vertical pasting of the triangles evaluated at $y\in\D(\bbone)$ is hence the fold map 
\[
\nabla\colon\colim_S\Delta_S(y)=\coprod_{s\in S} y\to y.
\]
An evaluation of the vertical pasting of the squares at $X\in\D(S^\rhd)$ is the natural transformation $i_S^\ast\ndia_{S^\rhd}(X)\to\Delta_S(X_\infty)$ induced by the structure maps of $X$. Thus, the total pasting applied to $X$ yields the map $\coprod_{s\in S}X_s\to X_\infty$ induced from the underlying diagram $\ndia_{S^\rhd}X$, i.e., the map detecting if $\ndia_{S^\rhd}X$ is a coproduct cocone in the usual sense. Since the upper two natural transformations are invertible, this is the case if and only if the pasting of the lower two natural transformations is an isomorphism on $X$. Note that this latter pasting is the canonical mate associated to the square on the left in \eqref{eq:htpy-exact-cocone-2} in the special case of $A=S$, which by \autoref{prop:cone} is an isomorphism if and only if $X\in\D(S^\rhd)$ is a coproduct cocone.
\end{proof}

\begin{defn}
Let $S$ be a discrete category. A morphism of derivators \textbf{preserves $S$-fold coproducts} if it preserves left Kan extensions along $\nabla_S\colon S\to\bbone$. A morphism of derivators \textbf{preserves initial objects} if it preserves left Kan extensions along $\emptyset\colon\emptyset\to\bbone$.
\end{defn}
 
Dually, we speak about morphisms preserving products or terminal objects. These two notions reduce to the usual categorical notions as we show next.

\begin{prop}\label{prop:coprod-pres}
Let $F\colon\D\to\E$ be a morphism of derivators and let $S$ be a (possibly empty) discrete category. The following are equivalent.
\begin{enumerate}
\item The morphism $F$ preserves $S$-fold coproducts.
\item The morphism $F$ preserves left Kan extensions along $i_S\colon S\to S^\rhd.$
\item The functor $F_{S^\rhd}\colon\D(S^\rhd)\to\E(S^\rhd)$ sends coproduct cocones to coproduct cocones.
\item The underlying functor $F_\bbone\colon\D(\bbone)\to\E(\bbone)$ preserves $S$-fold coproducts.
\item Every functor $F_A\colon\D(A)\to\E(A),A\in\cCat,$ preserves $S$-fold coproducts.
\end{enumerate}
\end{prop}
\begin{proof}
The equivalence of the first three statements is simply a special case of \autoref{prop:mor-colim-vs-cocone}. We next show that the first and fourth statement are equivalent. Let $F\colon\D\to\E$ be a morphism of derivators and let us consider the following two pastings
\begin{equation}
  \vcenter{\xymatrix{
      \D(S)\ar[r]^{F_S}&
      \E(S)\ar[r]^-{\ndia_S}_-\simeq \dltwocell\omit{\cong}&
      \E(\bbone)^S\dltwocell\omit{\id}\\
      \D(\bbone)\ar[r]_-{F_\bbone}\ar[u]^-{\nabla_S^\ast} &
      \E(\bbone)\ar[r]_-\id\ar[u]^-{\nabla_S^\ast} &
      \E(\bbone),\ar[u]^-{\Delta_S}
    }}
  \qquad
  \vcenter{\xymatrix{
      \D(S)\ar[r]^-{\ndia_S}_-\simeq&
      \D(\bbone)^S\ar[r]^-{F_\bbone^S}\dltwocell\omit{\id} &
      \E(\bbone)^S\dltwocell\omit{\id}\\
      \D(\bbone)\ar[r]_-\id\ar[u]^{\nabla_S^\ast}  &
      \D(\bbone)\ar[r]_-{F_{\bbone}}\ar[u]^{\Delta_S} &
      \E(\bbone),\ar[u]_{\Delta_S}
    }}
\end{equation}
in which the two inner squares commute. Note that these inner squares are horizontally constant (the horizontal functors are equivalences and the squares are populated by natural isomorphisms) and they hence have invertible mates. These two pastings agree up to a vertical pasting by the pseudo-naturality constraint of the underlying diagram morphism. Putting this together, $F$ preserves $S$-fold coproducts if and only if the canonical mate of the pasting on the left is an isomorphism if and only if the canonical mate of the pasting on the right is an isomorphism if and only if the underlying functor $F_\bbone\colon\D(\bbone)\to\E(\bbone)$ preserves $S$-fold coproducts, thereby establishing the equivalence of (i) and (iv). Since (v) clearly implies (iv), it remains to show that (i) implies (v). But, using the equivalence of (i) and (iv), this is an immediate consequence of \autoref{cor:shift-cocont}.
\end{proof}

\autoref{prop:coprod-pres} applies, in particular, in the case of $S=\emptyset$, thereby yielding statements about morphisms of derivators which preserve initial objects. As an immediate consequence we obtain the following result.

\begin{figure}
\centering
\[\xymatrix{
\D(\bbone)^{S^\rhd}\ar[r]^-{i_S^\ast}&\D(\bbone)^S\ar@/^1.0pc/[dr]^-{\colim_S}&\\
\D(S^\rhd)\ar[r]^-{i_S^\ast}\ar[u]^-{\ndia_{S^\rhd}}\urtwocell\omit{\id}&\D(S)\ar[u]^-\simeq_-{\ndia_S}\ar[r]^-{(\pi_S)_!}\urtwocell\omit{\cong}&\D(\bbone)\\
\D(S^\rhd)\ar[r]_-{\infty^\ast}\ar[u]^-\id\urtwocell\omit{}&\D(\bbone)\ar@/_1.0pc/[ur]_-\id\ar[u]_-{\pi_S^\ast}\urtwocell\omit{\varepsilon}&
}
\]
\caption{Matching of categorical coproducts and homotopy coproducts.}
\label{fig:coprod}
\end{figure}

\begin{cor}\label{cor:coprod-pres}
A morphism of derivators which preserves binary coproducts also preserves non-empty, finite coproducts. 
\end{cor}
\begin{proof}
This is immediate from the corresponding result in ordinary category theory and two applications of \autoref{prop:coprod-pres}.
\end{proof}

\begin{lem}\label{lem:initial-cosieve}
A morphism of derivators preserves initial objects if and only if it preserves left Kan extensions along cosieves.
\end{lem}
\begin{proof}
If a morphism preserves initial objects, then it also preserves left Kan extensions along cosieves by \cite[Prop.~1.23]{groth:ptstab} and \autoref{lem:ff-kan-ess-im}. For the converse direction it suffices to consider the empty cosieve $\emptyset\colon\emptyset\to\bbone$.
\end{proof}

In order to apply \autoref{cor:coprod-pres} it is convenient to have a different description of binary coproducts in derivators. Recall from ordinary category theory that coproducts of two objects $x,y\in\cC$ in a category with finite coproducts can equivalently be described by pushout diagrams
\[
\xymatrix{
\emptyset\ar[r]\ar[d]&x\ar[d]\\
y\ar[r]&x\sqcup y\pushoutcorner
}
\]
in which $\emptyset$ is an initial object. To extend this to derivators let us consider the functor
\begin{equation}\label{eq:copr-I}
k=((1,0),(0,1))\colon\bbone\sqcup\bbone\to\square=[1]^2
\end{equation}
which factors as compositions of fully faithful functors
\begin{equation}\label{eq:copr-II}
\bbone\sqcup\bbone\stackrel{i}{\to}\ulcorner\stackrel{i_\ulcorner}{\to}\square\qquad\text{and}\qquad\bbone\sqcup\bbone\stackrel{j}{\to}\lrcorner\stackrel{i_\lrcorner}{\to}\square.
\end{equation}
Here, $i_\ulcorner\colon\ulcorner\to\square$ and $i_\lrcorner\colon\lrcorner\to\square$ denote the inclusions of the full subcategories obtained by removing the final object $(1,1)$ and initial object $(0,0)$, respectively. Our naming convention for the objects in $\square$ is 
\[
\xymatrix{
(0,0)\ar[r]\ar[d]&(1,0)\ar[d]\\
(0,1)\ar[r]&(1,1).
}
\]

Given a derivator \D, as a special case of \autoref{defn:colim-cocone}, a diagram $X\in\D(\lrcorner)$ is a coproduct cocone if it lies in the essential image of $j_!\colon\D(\bbone\sqcup\bbone)\to\D(\lrcorner)$.

\begin{notn}\label{notn:copr-square}
Let \D be a derivator. We denote by $\D(\square)^\mathrm{copr}\subseteq\D(\square)$ the full subcategory spanned by the cocartesian squares $X$ such that $X_{(0,0)}\cong\emptyset$, and we refer to any object in $\D(\square)^\mathrm{copr}$ as a \textbf{coproduct square}.
\end{notn}

The justification for this terminology is provided by the following lemma.

\begin{lem}\label{lem:copr-square}
For every derivator \D the left Kan extension along $k$ \eqref{eq:copr-I} induces an equivalence $\D(\bbone\sqcup\bbone)\simeq\D(\square)^{\mathrm{copr}}$. Moreover, a square $X$ lies in $\D(\square)^{\mathrm{copr}}$ if and only if $X_{(0,0)}\cong\emptyset$ and if the restriction $i_\lrcorner^\ast X\in\D(\lrcorner)$ is a coproduct cocone.
\end{lem}
\begin{proof}
The functor \eqref{eq:copr-I} is fully faithful hence so is $k_!\colon\D(\bbone\sqcup\bbone)\to\D(\square)$. Since $k$ factors as indicated in \eqref{eq:copr-II}, there are by \autoref{lem:lkan-commute-naive} natural isomorphisms
\[
k_!\cong (i_\ulcorner)_!i_!\cong (i_\lrcorner)_! j_!.
\]
All of these functors are fully faithful and these factorizations yield two different descriptions of the essential image of $k_!$. Using the natural isomorphism $k_!\cong(i_\ulcorner)_!i_!$ we see that $X\in\D(\square)$ lies in the essential image of $k_!$ if and only if $X$ is cocartesian and $X_{(0,0)}\cong\emptyset$, i.e., if and only if $X$ is a coproduct square. In fact, since $i$ is a cosieve this follows from \cite[Prop.~1.23]{groth:ptstab}. Similarly, using the isomorphism $k_!\cong (i_\lrcorner)_!j_!$ and the fact that $i_\lrcorner$ is a cosieve it follows from the same proposition that the essential image of $k_!$ consists precisely of those $X$ with $X_{(0,0)}\cong \emptyset$ and such that $i_\lrcorner^\ast X$ is a coproduct cocone.
\end{proof}

This seemingly picky discussion allows us in \S\ref{sec:exact-htpy-finite} to show that right exact morphisms preserve finite coproducts.

\section{Pointed morphisms}
\label{sec:pointed}

In this section we define pointed morphisms of pointed derivators. It is is shown that this is a natural class of morphisms which allows for canonical comparison maps expressing a lax or an oplax compatibility with suspensions, loops, cofibers, fibers, and similar constructions.

\begin{lem}\label{lem:mor-pointed-zero}
A morphism of pointed derivators preserves initial objects if and only if it preserves terminal objects.
\end{lem}
\begin{proof}
This is immediate from \autoref{prop:coprod-pres}.
\end{proof}

\begin{cor}\label{cor:initial-cosieve}
A morphism of pointed derivators preserves zero objects if and only if it preserves left extensions by zero if and only if it preserves right extensions by zero.
\end{cor}
\begin{proof}
This is immediate from \autoref{lem:initial-cosieve}.
\end{proof}

\begin{rmk}\label{rmk:pointed}
To put this in words, a morphism between pointed derivators which preserves initial objects (`a construction on the left') preserves right Kan extensions along sieves (`a construction on the right'). More interestingly, this phenomenon reappears in the stable context, and we get back to this in \S\ref{sec:exact-htpy-finite} and~\cite{groth:char}.
\end{rmk}

\begin{defn}\label{defn:pointed-mor}
A morphism of pointed derivators is \textbf{pointed} if it preserves zero objects.
\end{defn}

\begin{egs}\label{egs:pointed}
Equivalences, left adjoint, and right adjoint morphisms of pointed derivators are pointed. The class of pointed morphisms is closed under compositions, shifting, products, and the passage to opposite or equivalent morphisms.
\end{egs}

\begin{rmk}
Despite being very simple to prove, \autoref{cor:initial-cosieve} is quite useful. In the framework of derivators, one often combines three types of constructions, namely restriction morphisms, left Kan extension morphisms, and right Kan extension morphisms. In the context of a pointed morphism, the above corollary frequently allows us to construct canonical comparison maps between suitable combinations of such constructions.
\begin{enumerate}
\item Pointed morphisms are compatible with restrictions up to specified isomorphisms by pseudo-naturality.
\item Pointed morphisms preserve left and right extensions by zero up to canonical isomorphisms (\autoref{cor:initial-cosieve}).
\item Pointed morphisms admit canonical, possibly non-invertible comparison maps for more general Kan extensions (see \eqref{eq:mate-mor-cocont} and \eqref{eq:mate-mor-cont}). 
\end{enumerate}
Being able to pass to inverses of the first two kinds of maps often allows us to replace certain zigzags of morphisms by direct morphisms.\end{rmk}

 We illustrate this by three closely related examples.

\begin{con}\label{con:pointed-cof-fib}
Let $F\colon\D\to\E$ be a pointed morphism of pointed derivators. We construct canonical, possibly non-invertible natural transformations populating the following squares,
\[
\xymatrix{
\D([1])\ar[r]^-\cof\ar[d]_-F&\D([1])\ar[d]^-F&&
\D([1])\ar[r]^-\fib\ar[d]_-F\drtwocell\omit{}&\D([1])\ar[d]^-F\\
\E([1])\ar[r]_-\cof&\E([1]),\ultwocell\omit{}&&
\E([1])\ar[r]_-\fib&\E([1]).
}
\]
By duality it is enough to construct the natural transformation in the square on the left. Denoting by $i\colon[1]\to\ulcorner$ the sieve classifying the horizontal morphism $(0,0)\to (1,0)$, let us recall that $\cof$ is defined by the rows in the following diagram
\[
\xymatrix{
\D([1])\ar[r]^-{i_\ast}\ar[d]_-F\drtwocell\omit{\cong}&\D(\ulcorner)\ar[r]^-{(i_\ulcorner)_!}\ar[d]_-F&\D(\square)\ar[r]^-{(k')^\ast}\ar[d]^-F&\D([1])\ar[d]^-F\\
\E([1])\ar[r]_-{i_\ast}&\E(\ulcorner)\ar[r]_-{(i_\ulcorner)_!}&\E(\square)\ar[r]_-{(k')^\ast}\ultwocell\omit{}&\E([1]),\ultwocell\omit{\cong}
}
\]
in which $k'\colon[1]\to\square$ classifies the vertical morphism $(1,0)\to(1,1)$. In this diagram the two natural transformations on the left are instances of \eqref{eq:mate-mor-cont} and \eqref{eq:mate-mor-cocont}, respectively, while the transformation on the right is a pseudo-naturality isomorphism. As a pointed morphism, $F$ preserves by \autoref{cor:initial-cosieve} right Kan extensions along the sieve $i$. Passing to the inverse of the natural transformation on the left, we can define the canonical transformation as the following pasting
\begin{align}\label{eq:pointed-cof}
\cof\circ F_{[1]}&= (k')^\ast\circ (i_\ulcorner)_!\circ i_\ast\circ  F_{[1]}\\
&\cong (k')^\ast\circ (i_\ulcorner)_!\circ  F_\ulcorner\circ  i_\ast\\
&\to (k')^\ast\circ  F_\square\circ (i_\ulcorner)_!\circ  i_\ast\\
&\cong F_{[1]}\circ (k')^\ast\circ  (i_\ulcorner)_!\circ  i_\ast\\
&= F_{[1]}\circ \cof.
\end{align}
The functoriality of mates with respect to pasting implies that these canonical natural transformations
\begin{equation}\label{eq:pointed-cof-fib}
\cof\circ F\to F\circ\cof\qquad\text{and}\qquad F\circ\fib\to \fib\circ F
\end{equation}
are compatible with compositions and identities. If the transformations in \eqref{eq:pointed-cof-fib} are invertible, then we say that $F$ \textbf{preserves cofibers or fibers}, respectively. 
\end{con}

\begin{con}\label{con:pointed-susp-omega}
Let us recall that the suspension functor in a pointed derivator \D is defined as
\[
\Sigma=(1,1)^\ast\circ (i_\ulcorner)_!\circ (0,0)_\ast\colon\D(\bbone)\to\D(\ulcorner)\to\D(\square)\to\D(\bbone).
\]
Since $(0,0)\colon\bbone\to\ulcorner$ is a sieve, associated to a pointed morphism $F\colon\D\to\E$ of pointed derivators we can consider the composition
\begin{align}
\Sigma\circ F_\bbone&= (1,1)^\ast\circ (i_\ulcorner)_!\circ (0,0)_\ast\circ  F_\bbone\\
&\cong (1,1)^\ast\circ (i_\ulcorner)_!\circ  F_\ulcorner\circ  (0,0)_\ast\\
&\to (1,1)^\ast\circ  F_\square\circ (i_\ulcorner)_!\circ  (0,0)_\ast\\
&\cong F_\bbone\circ (1,1)^\ast\circ  (i_\ulcorner)_!\circ  (0,0)_\ast\\
&= F_\bbone\circ \Sigma.
\end{align}
This leads to canonical natural transformations
\begin{equation}\label{eq:pointed-sigma-omega}
\Sigma\circ F\to F\circ\Sigma\qquad\text{and}\qquad F\circ\Omega\to \Omega\circ F, 
\end{equation}
which are compatible with respect to compositions and identities. And if these transformations are invertible, the morphism is said to \textbf{preserve suspensions or loops}, respectively.
\end{con}

\begin{rmk}
Note that the canonical morphism \eqref{eq:pointed-sigma-omega} points in the opposite direction as the natural transformations $\varphi\colon \Sigma F\to F\Sigma$ yielding \emph{exact structures} on additive functors $F$ of triangulated categories. We will come back to this in \S\ref{sec:can-triang}.
\end{rmk}

\begin{con}\label{con:ptd-cof-seq}
Let $[2]$ be the poset $(0<1<2)$ and let $\boxbar=[2]\times[1]$. The formation of coherent cofiber sequences in a pointed derivator \D defines a functor $\D([1])\to\D(\boxbar)$. Since this functor is given by a right extension by zero followed by a left Kan extensions, associated to a pointed morphism $F\colon\D\to\E$ of pointed derivators there is a canonical natural transformation populating the diagram
\begin{equation}\label{eq:mor-pres-cof-seq}
\vcenter{
\xymatrix{
\D([1])\ar[r]\ar[d]_-F&\D(\boxbar)\ar[d]^-F\\
\E([1])\ar[r]&\E(\boxbar).\ultwocell\omit{}
}
}
\end{equation}
These natural transformations are compatible with identities and composition. A pointed morphism \textbf{preserves cofiber sequences} if this canonical transformation is invertible, and there is the dual notion of a pointed morphism which \textbf{preserves fiber sequences}.
\end{con}

\begin{rmk}\label{rmk:comp-of-mates}
As a consequence of \autoref{lem:mates-cont-trafo}, the canonical transformations constructed in \autoref{con:pointed-cof-fib}, \autoref{con:pointed-susp-omega}, and \autoref{con:ptd-cof-seq} are compatible with natural transformations of pointed morphisms. 
\end{rmk}

For later reference we make this remark explicit in the following special case.

\begin{prop}\label{prop:ptd-comp-susp-natural}
For pointed morphisms $F,G\colon\D\to\E$ of pointed derivators and a natural transformation $\alpha\colon F\to G$ the following diagram commutes
\[
\xymatrix{
\Sigma\circ F\ar[r]\ar[d]_-\alpha&F\circ \Sigma\ar[d]^-\alpha\\
\Sigma \circ G\ar[r]&G\circ \Sigma,
}
\]
in which the unlabeled morphisms are instances of \eqref{eq:pointed-sigma-omega}.
\end{prop}
\begin{proof}
We write $i=(0,0)\colon\bbone\to\ulcorner$ for the sieve classifying the initial object and consider the following diagram
\[
\xymatrix@-0.4pc{
(1,1)^\ast(i_\ulcorner)_!i_\ast F_\bbone\ar[d]_-\alpha&(1,1)^\ast(i_\ulcorner)_!F_\ulcorner i_\ast\ar[r]\ar[l]_-\cong\ar[d]_-\alpha&(1,1)^\ast F_\square(i_\ulcorner)_!i_\ast\ar[d]^-\alpha\ar[r]^-\cong&F_\bbone(1,1)^\ast (i_\ulcorner)_!i_\ast\ar[d]^-\alpha\\
(1,1)^\ast(i_\ulcorner)_!i_\ast F_\bbone&(1,1)^\ast(i_\ulcorner)_!F_\ulcorner i_\ast\ar[r]\ar[l]^-\cong&(1,1)^\ast F_\square(i_\ulcorner)_!i_\ast\ar[r]_-\cong&F_\bbone(1,1)^\ast (i_\ulcorner)_!i_\ast,
}
\]
The rows coincide with the canonical transformations \eqref{eq:pointed-sigma-omega} from \autoref{con:pointed-susp-omega}, hence it suffices to show that this diagram commutes. The two squares on the left commute by two applications of \autoref{lem:mates-cont-trafo} while the square on the right commutes by the coherence properties of a modification.
\end{proof}

\section{Exact morphisms and homotopy finite Kan extensions}
\label{sec:exact-htpy-finite}

In this section we collect some results concerning left exact, right exact, and exact morphisms of derivators. We show that right exact morphisms preserve many basic constructions and, more generally, left homotopy finite left Kan extensions.

\begin{defn}\label{defn:exact-mor}
\begin{enumerate}
\item A morphism of derivators \textbf{preserves pushouts} if it preserves colimits of shape $\ulcorner$. Dually, a morphism of derivators \textbf{preserves pullbacks} if it preserves limits of shape $\lrcorner$.
\item A morphism of derivators is \textbf{right exact} if it preserves initial objects and pushouts. Dually, a morphism of derivators is \textbf{left exact} if it preserves terminal objects and pullbacks. 
\item A morphism of derivators is \textbf{exact} if it is right exact and left exact.
\end{enumerate}
\end{defn}

\begin{prop}[{\cite[Cor.~4.17]{groth:ptstab}}]\label{prop:rex-fin-cop}
Right exact morphisms of derivators preserve finite coproducts and left exact morphisms preserve finite products.
\end{prop}
\begin{proof}
Let $F\colon\D\to\E$ be a right exact morphism of derivators. We show that $F$ preserves binary coproducts. Let $k\colon\bbone\sqcup\bbone\to\square$ be the functor classifying the objects $(1,0),(0,1)$, and let 
\begin{equation}\label{eq:rex-fin-cop}
\vcenter{
\xymatrix{
\bbone\sqcup\bbone\ar[r]^-i\ar[d]_-j&\ulcorner\ar[d]^-{i_\ulcorner}\\
\lrcorner\ar[r]_-{i_\lrcorner}&\square
}
}
\end{equation}
be the two factorizations of $k$ as in \eqref{eq:copr-II}. Since $i\colon\bbone\sqcup\bbone\to\ulcorner$ is a cosieve, the functor $i_!$ is by \cite[Prop.~1.23]{groth:ptstab} left extension by initial objects. As $F$ preserves initial objects, by \autoref{lem:initial-cosieve} it also preserves left Kan extensions along $i$. As a right exact morphism, $F$ also preserves left Kan extensions along $i_\ulcorner$ (\autoref{prop:mor-colim-vs-cocone}), and $F$ hence also preserves left Kan extensions along $k=i_\ulcorner\circ i$ (\autoref{prop:sorite-lim-II}). Since $i_\lrcorner$ and $j$ are fully faithful and $k=i_\lrcorner\circ j$, it follows that $F$ also preserves left Kan extensions along $j$ (\autoref{lem:cancel-lkan}), and \autoref{prop:coprod-pres} then implies that $F$ preserves binary coproducts. The case of non-empty, finite coproducts is taken care of by \autoref{cor:coprod-pres}. Moreover, $F$ also preserves empty coproducts, i.e., initial objects since $F$ is right exact. 
\end{proof}

We revisit \autoref{warn:cancel-lkan}.

\begin{warn}\label{warn:cancel-lkan-2}
Let $F\colon\D\to\E$ be a morphism of derivators, let $u\colon A\to B$ and $v\colon B\to C$ be fully faithful. If $F$ preserves left Kan extensions along $vu$, then, in general, $F$ does not preserve left Kan extensions along $v$.

To construct a counter-example we again consider the situation in \eqref{eq:rex-fin-cop}. Let $F\colon\D\to\E$ be a morphism of derivators which preserves finite, possibly empty coproducts. By \autoref{lem:copr-square} and \autoref{prop:coprod-pres} the morphism $F$ preserves the essential image of $k_!$, and hence also left Kan extensions along $k$ by \autoref{lem:ff-kan-ess-im}. However, in general, $F$ does not preserve left Kan extensions along $i_\ulcorner$ or, equivalently, pushouts (\autoref{prop:mor-colim-vs-cocone}). In fact, any ordinary functor between complete and cocomplete categories which preserves finite coproducts but not pushouts provides a counter-example.
\end{warn}

\begin{prop}[{\cite[Prop.~3.21]{groth:ptstab}}]\label{prop:rex-susp}
Right exact morphisms of pointed derivators preserve suspensions.
\end{prop}
\begin{proof}
Given a right exact morphism $F\colon\D\to\E$ of pointed derivators, we have to show that the canonical transformation $\Sigma\circ F\to F\circ\Sigma$ \eqref{eq:pointed-sigma-omega} is invertible. But this is immediate from \autoref{con:pointed-susp-omega}, since right exact morphisms preserve left Kan extensions along $i_\ulcorner$ (\autoref{prop:mor-colim-vs-cocone}).
\end{proof}

In a similar way one makes precise and proves the following result.

\begin{prop}\label{prop:rex-basic-con}
Right exact morphisms of pointed derivators preserve cones, cofibers, cofiber squares, cofiber sequences, and iterated cofiber sequences.
\end{prop}

\begin{rmk}
The canonical isomorphisms in \autoref{prop:rex-susp} and \autoref{prop:rex-basic-con} are compatible with natural transformations of right exact morphisms (\autoref{rmk:comp-of-mates} and \autoref{prop:ptd-comp-susp-natural}).
\end{rmk}

\begin{prop}[{\cite[Cor.~4.17]{groth:ptstab}}]\label{prop:exact}
A morphism of stable derivators is left exact if and only if it right exact if and only if it is exact.
\end{prop}
\begin{proof}
By duality it suffices to show that a left exact morphism of stable derivators is right exact, and by \autoref{lem:mor-pointed-zero} it only remains to show that a left exact morphism preserves pushouts. Since the morphism preserves pullbacks, it sends cartesian squares to cartesian squares (\autoref{prop:mor-colim-vs-cocone}), which, using stability, is to say that it sends cocartesian squares to cocartesian squares. By an additional application of \autoref{prop:mor-colim-vs-cocone} this implies that the morphism preserves pushouts.
\end{proof}

\begin{cor}\label{cor:adjoint-exact}
Left adjoint morphisms, right adjoint morphisms, and equivalences of stable derivators are exact.
\end{cor}
\begin{proof}
This follows from \autoref{prop:exact} and the (co)continuity of adjoints.
\end{proof}

This applies, in particular, to derived adjunctions and equivalences arising from Quillen adjunctions and Quillen equivalences of stable model categories.

\begin{egs}\label{egs:exact-mor}
The class of exact morphisms is closed under compositions, shifting, products, and the passage to opposite, naturally isomorphic or equivalent morphisms. \end{egs}

As a preparation for the compatibility of right exact morphism with left homotopy finite left Kan extensions we recall a theorem of Ponto--Shulman~\cite{ps:linearity} showing that such morphisms preserve homotopy finite colimits. 

\begin{defn}
A small category~$A$ is \textbf{strictly homotopy finite} if the nerve $NA$ is a finite simplicial set. A small category is \textbf{homotopy finite} if it is equivalent to a strictly homotopy finite category.
\end{defn}

Thus, strictly homotopy finite categories are precisely the finite and skeletal categories which have no non-trivial endomorphisms. Alternatively, these are precisely the finite and directed category. (Let us recall that a category is \textbf{directed} if whenever $f\colon a\to b$ is a non-identity morphism, then there is no morphism $b\to a$.)

\begin{thm}[{\cite[Thm.~7.1]{ps:linearity}}]\label{thm:right-exact-finite}
Every right exact morphism of derivators preserves homotopy finite colimits.
\end{thm}

With our preparation we now extend this result to sufficiently finite left Kan extensions (\autoref{thm:rex-general}). The following definition is inspired by \autoref{rmk:cocont-colimits} and the notion of \emph{$L$-finite limits} in classical category theory \cite[Prop.~7]{pare:simply-conn-lim}.

\begin{defn}\label{defn:left-homotopy-finite}
Let $u\colon A\to B$ be a functor between small categories.
\begin{enumerate}
\item The functor $u$ is \textbf{left homotopy finite} if for every $b\in B$ there is a homotopy finite category~$C_b$ and a homotopy final functor $C_b\to (u/b)$.
\item The functor $u$ is \textbf{right homotopy finite} if for every $b\in B$ there is a homotopy finite category~$C_b$ and a homotopy cofinal functor $C_b\to (b/u)$.
\end{enumerate}
\end{defn}

We say that a morphism of derivators \textbf{preserves left homotopy finite left Kan extensions} if it preserves left Kan extensions along all left homotopy finite functors.

\begin{thm}\label{thm:rex-general}
Every right exact morphism of derivators preserves left homotopy finite left Kan extensions.
\end{thm}
\begin{proof}
Let $F\colon\D\to\E$ be a right exact morphism of derivators and let $u\colon A\to B$ be a left homotopy finite functor. By~\autoref{rmk:cocont-colimits} it is enough to show that $F$ preserves colimits of shape $(u/b)$ for all $b\in B$. By~\autoref{defn:left-homotopy-finite} for every $b\in B$ there is a homotopy finite category~$C_b$ and a homotopy final functor $i_b\colon C_b\to (u/b)$. \autoref{thm:right-exact-finite} shows that $F$ preserves colimits of shape~$C_b$ and, by~\autoref{lem:final}, the same is true for colimits of shape~$(u/b)$.
\end{proof}

\begin{rmk}
By this theorem right exact morphisms of derivators preserve a large class of left Kan extensions. It turns out that such morphisms also preserve many canonical isomorphisms between expressions involving such right Kan extensions, and we intend to come back to this in \cite{groth:formal}. In particular, in the framework of stable derivators this leads to a calculus of uniform formulas for stable derivators.
\end{rmk}

Also the following two variants of \autoref{thm:rex-general} are convenient. Due to their importance, we state them as separate theorems.

\begin{thm}\label{thm:rex-pointed}
Every right exact morphism of pointed derivators preserves
\begin{enumerate}
\item left homotopy finite left Kan extensions and
\item right extensions by zero.
\end{enumerate}
\end{thm}
\begin{proof}
This is immediate from~\autoref{thm:rex-general} and \autoref{cor:initial-cosieve}.
\end{proof}

\begin{thm}\label{thm:exact}
Every exact morphism of stable derivators preserves
\begin{enumerate}
\item left homotopy finite left Kan extensions,
\item left extensions by zero,
\item right homotopy finite right Kan extensions, and 
\item right extensions by zero.
\end{enumerate}
\end{thm}
\begin{proof}
This is immediate from~\autoref{thm:rex-pointed} and its dual.
\end{proof}

\begin{rmk}
\begin{enumerate}
\item Of course, the subcases (ii) and (iv) in \autoref{thm:exact} are redundant but are mentioned in order to emphasize them. 
\item These three theorems apply rather frequently since many typical constructions satisfy the above finiteness assumptions; see for example \cite{gst:basic,gst:tree,gst:Dynkin-A,gst:acyclic}. Additional applications of these theorems will also appear in the sequel \cite{groth:char}.
\item Let $\cDER_{\mathrm{St,ex}}$ be the $2$-category of stable derivators, exact morphisms, and all natural transformations, and let $u\colon A\to B$ be a left homotopy finite functor. By \autoref{thm:exact} the lax natural transformation 
\[
u_!\colon(-)^A\to(-)^B\colon\cDER_{\mathrm{St,ex}}\to\cDER_{\mathrm{St,ex}}
\]
which is the restriction of the lax natural transformation from \autoref{lem:lax-lkan} to $\cDER_{\mathrm{St,ex}}$ is a pseudo-natural transformation. There is variant of this for right Kan extensions along right homotopy finite functors.
\end{enumerate}
\end{rmk}

\section{The canonicity of canonical triangulations}
\label{sec:can-triang}

The values of strong, stable derivators can be turned into triangulated categories, and we refer to these triangulations as \textbf{canonical triangulations}; see \cite{franke:adams,maltsiniotis:seminar} or \cite[\S4.2]{groth:ptstab}. Our next goal is to show that these triangulations are $2$-functorial with respect to exact morphisms and arbitrary transformations. 

Given a strong, stable derivator, this specializes to canonical exact structures on restriction and Kan extension functors, and there is also a $2$-categorical variant of this statement. These results lead to a uniqueness statement for canonical triangulations, thereby justifying the terminology. Moreover, there are variants for canonical higher triangulations \cite{beilinson:perverse,maltsiniotis:higher,gst:Dynkin-A}.

\begin{defn}\label{defn:exact-functor}
Let $\cT$ and $\cT'$ be triangulated categories. An \textbf{exact functor} $\cT\to\cT'$ is a pair $(F,\sigma)$ consisting of 
\begin{enumerate}
\item an additive functor $F\colon\cT\to\cT'$ and 
\item a natural transformation $\sigma\colon F\Sigma\to\Sigma F$
\end{enumerate}
such that for every distinguished triangle $X\stackrel{f}{\to} Y\stackrel{g}{\to} Z\stackrel{h}{\to}\Sigma X$ in~$\cT$ the image triangle $FX\stackrel{Ff}{\to} FY\stackrel{Fg}{\to} FZ\stackrel{\sigma\circ Fh}{\to}\Sigma FX$ is distinguished in~$\cT'$. The natural transformation $\sigma$ is an \textbf{exact structure} on $F$.
\end{defn}

\begin{con}\label{con:exact-exact}
For every exact morphism $F\colon\D\to\E$ of stable derivators and $A\in\cCat$ we construct a natural isomorphism 
\begin{equation}\label{eq:exact-str}
\sigma=\sigma_A\colon F_A\circ\Sigma\toiso\Sigma\circ F_A.
\end{equation}
Passing to shifted derivators, we can assume that $A=\bbone$ (\autoref{egs:exact-mor}). Recall that for every pointed morphism we constructed a canonical natural transformation \eqref{eq:pointed-sigma-omega} pointing in the opposite direction, which we know to be invertible for right exact morphisms (\autoref{prop:rex-susp}). We define \eqref{eq:exact-str} to be the inverse of \eqref{eq:pointed-sigma-omega} and refer to it as the \textbf{canonical exact structure} on $F_A$.
\end{con}

This terminology is justified by \autoref{thm:exact-exact}. As a preparation we make the following construction.

\begin{con}\label{con:comp-susp-sq}
Let \D be a pointed derivator and let $X\in\D(\square)$ be such that $X_{1,0}\cong X_{0,1}\cong 0$,
\[
\xymatrix{
X_{0,0}\ar[r]\ar[d]&0\ar[d]\\
0\ar[r]&X_{1,1}.
}
\]
Denoting by $i=(0,0)\colon\bbone\to\ulcorner$ the sieve classifying the initial object, it follows from \cite[Prop.~3.6]{groth:ptstab} that $i_\ulcorner^\ast X\in\D(\ulcorner)$ lies in the essential image of $i_\ast$ which is to say that the unit $\eta\colon i_\ulcorner^\ast X\to i_\ast i^\ast i_\ulcorner^\ast X$ is an isomorphism. This allows us to form the natural transformation
\[
\Phi\colon (i_\ulcorner)_!i_\ast i^\ast i_\ulcorner^\ast X\stackrel{\eta^{-1}}{\to}(i_\ulcorner)_!i_\ulcorner^\ast X\stackrel{\varepsilon}{\to} X
\]
on the full subcategory $\D(\square)^\ex\subseteq\D(\square)$ spanned by all diagrams satisfying this vanishing condition. Since the suspension is defined by $\Sigma=(1,1)^\ast(i_\ulcorner)_!(0,0)_\ast$, an evaluating at $(1,1)$ yields
\begin{equation}\label{eq:comp-susp-sq}
\varphi=\Phi_{1,1}\colon \Sigma (X_{0,0})\to X_{1,1}.
\end{equation}
By the fully faithfulness of $i_\ulcorner$, this transformation is invertible on \textbf{suspension squares}, i.e., on squares $X\in\D(\square)^\ex$ which additionally are cocartesian.
\end{con}

\begin{thm}\label{thm:exact-exact}
Let $F\colon\D\to\E$ be an exact morphism of strong, stable derivators and let $A\in\cCat$. The isomorphism \eqref{eq:exact-str} turns $F_A\colon\D(A)\to\E(A)$ into an exact functor with respect to canonical triangulations.
\end{thm}
\begin{proof}
The functor $F_A\colon\D(A)\to\E(A)$ is additive by \autoref{prop:rex-fin-cop} and it hence remains to show that \eqref{eq:exact-str} defines an exact structure. Passing to shifted derivators we assume without loss of generality that $A=\bbone$, and it suffices to show that $F=F_\bbone\colon\D(\bbone)\to\E(\bbone)$ and $\sigma=\sigma_\bbone\colon F\circ\Sigma\cong\Sigma\circ F$ as in \eqref{eq:exact-str} send \emph{standard} distinguished triangles in $\D(\bbone)$ to distinguished triangles in $\E(\bbone)$.

To this end, let $Q\in\D(\boxbar)$ be a coherent cofiber sequence and let us consider the corresponding standard triangle
\[
(0,0)^\ast Q\to (1,0)^\ast Q\to (1,1)^\ast Q\to (2,1)^\ast Q\stackrel{\varphi^{-1}}{\to} \Sigma (0,0)^\ast Q,
\]
where $\varphi$ is as in \eqref{eq:comp-susp-sq}. In order to show that the associated image triangle in $\E(\bbone)$ is distinguished we pass to the following diagram
\[
\xymatrix@-1.4pc{
F(0,0)^\ast Q\ar[r]&F(1,0)^\ast Q\ar[r]&F(1,1)^\ast Q\ar[r]&F(2,1)^\ast Q\ar[r]^-{\varphi^{-1}}&F\Sigma (0,0)^\ast Q\ar[r]^-\sigma&\Sigma F(0,0)^\ast Q\\
(0,0)^\ast FQ\ar[u]^-\gamma_-\cong\ar[r]&(1,0)^\ast FQ\ar[u]^-\gamma_-\cong\ar[r]&(1,1)^\ast FQ\ar[u]_-\gamma^-\cong\ar[r]&(2,1)^\ast FQ\ar[u]_-\gamma^-\cong\ar[r]_-{\varphi^{-1}}&\Sigma (0,0)^\ast FQ.\ar[ru]^-\cong_-{\Sigma\gamma}
}
\]
In this diagram the vertical morphisms are pseudo-naturality isomorphisms, and the three squares to the left hence commute. The remaining morphisms are instances of inverses of \eqref{eq:comp-susp-sq} and the claimed exact structure $\sigma$. In fact, we consider the inverse of \eqref{eq:comp-susp-sq} for the cofiber sequence $Q$ and also for $FQ$, which is again a cofiber sequence since $F$ is exact (\autoref{prop:rex-basic-con}). To conclude the proof it remains to show that the quadrilateral on the right commutes. Writing $X=\iota_{02}Q\in\D(\square)$ for the corresponding suspension square, this amounts to showing that
\begin{equation}\label{eq:exact-exact}
\vcenter{
\xymatrix{
F(1,1)^\ast X\ar[r]^-{\varphi^{-1}}&F\Sigma (0,0)^\ast X\ar[r]^-\sigma&\Sigma F(0,0)^\ast X\\
(1,1)^\ast FX\ar[u]^-\gamma_-\cong\ar[r]_-{\varphi^{-1}}&\Sigma (0,0)^\ast FX\ar[ru]^-\cong_-\gamma
}
}
\end{equation}
commutes. Unraveling the definition of $\varphi$ and $\sigma$ (compare to \autoref{con:comp-susp-sq}, \autoref{con:exact-exact}, and \autoref{con:pointed-susp-omega}), it suffices to show that \autoref{fig:exact-exact} commutes. 
  \begin{figure}
    \centering
    \[\xymatrix{
F(X_{1,1})\ar@{=}[d]&&F\Sigma(X_{0,0})\ar@{=}[d]\\
F(1,1)^\ast X\ar@{}[dr]|{=}&F(1,1)^\ast(i_\ulcorner)_!(i_\ulcorner)^\ast X\ar[l]_-\varepsilon^-\cong\ar[r]^-\eta_-\cong\ar@{}[dr]|{=}&F(1,1)^\ast(i_\ulcorner)_!i_\ast i^\ast i_\ulcorner^\ast X\\
(1,1)^\ast F X\ar[u]^-\gamma_-\cong&(1,1)^\ast F(i_\ulcorner)_!i_\ulcorner^\ast X\ar[l]_-\varepsilon^-\cong\ar[r]^-\eta_-\cong\ar[u]^-\gamma_-\cong\ar@{}[dr]|{=}&(1,1)^\ast F(i_\ulcorner)_!i_\ast i^\ast i_\ulcorner^\ast X\ar[u]_-\gamma^-\cong\\
(1,1)^\ast (i_\ulcorner)_! i_\ulcorner^\ast FX\ar[u]^-\varepsilon_-\cong\ar@{}[dr]|{=}\ar[r]_-\gamma^-\cong\ar[d]_-\eta^-\cong&(1,1)^\ast (i_\ulcorner)_!F i_\ulcorner^\ast X\ar[r]_-\eta^-\cong\ar[u]_-\cong\ar[d]_-\eta^-\cong&(1,1)^\ast (i_\ulcorner)_!F i_\ast i^\ast i_\ulcorner^\ast X\ar[u]^-\cong\ar[d]_-\cong\\
(1,1)^\ast (i_\ulcorner)_!i_\ast i^\ast i_\ulcorner^\ast FX\ar[r]_-\gamma^-\cong&(1,1)^\ast (i_\ulcorner)_! i_\ast i^\ast F i_\ulcorner^\ast X\ar[r]_-\gamma^-\cong&(1,1)^\ast (i_\ulcorner)_! i_\ast F i^\ast i_\ulcorner^\ast X\\
\Sigma\big( (FX)_{0,0}\big)\ar@{=}[u]&&\Sigma F(X_{0,0})\ar@{=}[u]
      }\]
    \caption{The remaining quadrilateral also commutes.}
    \label{fig:exact-exact}
  \end{figure}
In this diagram the squares labeled by an equality sign commute as naturality squares, while the remaining two squares commute by two applications of \autoref{lem:units-vs-morphisms}.
\end{proof}

\begin{rmk}
With the exception of the slightly more involved direct verification that \eqref{eq:exact-exact} commutes (by means of \autoref{fig:exact-exact}), the proof of \autoref{thm:exact-exact} is completely straightforward. It turns out that there is a way to formalize formulas relative to $2$-categories of derivators and this makes such direct verifications obsolete. We will come back to this in \cite{groth:formal}.
\end{rmk}

\begin{rmk}\label{rmk:triang-canonical}
This result offers the following justification of the terminology \emph{canonical} triangulations. Note that the construction of canonical triangulations depends on certain choices, for example on the choice of a suspension functor which in turn relies on the choice of certain Kan extension functors. Let \D be a strong stable derivator, let $A\in\cCat$, and let us consider the exact identity morphism $\id\colon\D\to\D$. It follows from the proof of \autoref{thm:exact-exact} that if we endow $\D(A)$ with two different canonical triangulations, then the identity functor $\id\colon\D(A)\to\D(A)$ can be turned into an exact isomorphism with respect to these triangulations. 

Unraveling definitions one observes that this exact structure is obtained by combining total mates of identity transformations, i.e., by those natural transformations showing that the independently chosen Kan extensions are isomorphic. And in this sense also this exact structure is canonical.
\end{rmk}

An additional justification of the terminology is given by the following result (see also \autoref{thm:triang-2-fun} for a more systematic variant).

\begin{cor}\label{cor:triang-can}
Let \D be a strong, stable derivator and let $u\colon A\to B$ be in $\cCat$. The functors $u^\ast\colon\D(B)\to\D(A),u_!\colon\D(A)\to\D(B),$ and $u_\ast\colon\D(A)\to\D(B)$ can be turned into exact functors with respect to canonical triangulations.
\end{cor}
\begin{proof}
The calculus of parametrized Kan extensions yields adjunctions of strong, stable derivators
\[
(u_!,u^\ast)\colon\D^A\rightleftarrows\D^B\qquad\text{and}\qquad(u^\ast,u_\ast)\colon\D^B\rightleftarrows\D^A,
\]
given by restriction and Kan extension morphisms. These three morphisms are exact by \autoref{cor:adjoint-exact} and the underlying functors can hence canonically be turned into exact functors (\autoref{thm:exact-exact}). 
\end{proof}

\autoref{thm:exact-exact} also has a variant for natural transformations, and to make it precise we recall the following definition.

\begin{defn}\label{defn:exact-trafo}
Let $F,G\colon\cT\to\cT'$ be exact functors between triangulated categories. A natural transformation $\alpha\colon F\to G$ is \textbf{exact} if the following diagram commutes,
\[
\xymatrix{
F\circ\Sigma\ar[r]\ar[d]_-\alpha&\Sigma\circ F\ar[d]^-\alpha\\
G\circ\Sigma\ar[r]&\Sigma\circ G.
}
\]
\end{defn}

At the level of derivators there is no corresponding concept, since this compatibility is automatic.

\begin{cor}\label{cor:exact-2-cell}
Let $F,G\colon\D\to\E$ be exact morphisms of strong stable derivators, let $\alpha\colon F\to G$ be a natural transformation, and let $A\in\cCat$. The natural transformation $\alpha_A\colon F_A\to G_A$ is exact with respect to the canonical exact structures \eqref{eq:exact-str}.
\end{cor}
\begin{proof}
Passing to shifted derivators, we can again assume that $A=\bbone$ and in this case the result is immediate from \autoref{prop:ptd-comp-susp-natural}.
\end{proof}

\begin{cor}\label{cor:triang-can-II}
Let \D be a strong stable derivator and let $\alpha\colon u\to v$ be a natural transformation in $\cCat$. The induced transformation $\alpha^\ast\colon u^\ast\to v^\ast$ between $u^\ast,v^\ast\colon\D(B)\to\D(A)$ is exact with respect to the canonical exact structures from \autoref{cor:triang-can}.
\end{cor}
\begin{proof}
There is a natural transformation $\alpha^\ast\colon u^\ast\to v^\ast$ between the exact restriction morphisms $u^\ast,v^\ast\colon\D^B\to\D^A$. Hence, the statement follows from \autoref{cor:exact-2-cell} and an evaluation at the category~$A=\bbone$.
\end{proof}

The exact structures constructed in \autoref{thm:exact-exact} are functorial in exact morphisms. To formulate this in a special case more concisely, we denote by $\cTriaCAT$ the $2$-category of triangulated categories, exact functors, and exact natural transformations. This $2$-category comes with a forgetful $2$-functor $\cTriaCAT\to\cCAT.$

\begin{thm}\label{thm:triang-2-fun}
Every strong, stable derivator $\D\colon\cCat\op\to\cCAT$ admits a lift against the forgetful $2$-functor $\cTriaCAT\to\cCAT,$
\[
\xymatrix{
&\cTriaCAT\ar[d]\\
\cCat\op\ar[r]_-\D\ar@{-->}[ru]^-{\exists\D}&\cCAT,
}
\]
given by endowing $\D(A),A\in\cCat,$ with canonical triangulations.
\end{thm}
\begin{proof}
For every $A\in\cCat$ we choose a canonical triangulation on $\D(A)$ \cite[\S4.2]{groth:ptstab}. Given a functor $u\colon A\to B$, the restriction functor $u^\ast\colon\D(B)\to\D(A)$ is exact (\autoref{cor:triang-can}). In fact, we endow it with the canonical exact structure constructed in the proof of \autoref{thm:exact-exact}, while we choose $\id^\ast\colon\D(A)\to\D(A)$ to be endowed with the trivial exact structure. The definition of composition of exact functors of triangulated categories and the functoriality of mates with respect to pasting show that this construction is compatible with compositions. \autoref{cor:triang-can-II} concludes the proof.
\end{proof}

\begin{rmk}
These lifts of strong, stable derivators to the $2$-category $\cTriaCAT$ are themselves $2$-functorial. In fact, if we denote by $\cDER_{\mathrm{St,strong,ex}}$ the $2$-category of strong, stable derivators, exact morphisms, and all natural transformations, then choosing canonical triangulations for all strong, stable derivators yields a $2$-functor
\[
\cDER_{\mathrm{St,strong,ex}}\to\cTriaCAT^{\cCat\op}.
\]
Here, $\cTriaCAT^{\cCat\op}$ denotes the $2$-category of $2$-functors $\cCat\op\to\cTriaCAT$, exact pseudo-natural transformations, and exact modifications.
\end{rmk}

\begin{rmk}\label{rmk:triang-concise}
There are variants of the results of this section for canonical strong triangulations (\cite{beilinson:perverse,maltsiniotis:higher,gst:Dynkin-A}). In particular, a strong stable derivator also admits a lift against the forgetful functor from the $2$-category of strongly triangulated categories, exact functors, and exact natural transformations. The details are very similar to the case of ordinary triangulations and are left to the interested reader.
\end{rmk}

\begin{rmk}
Let us recall that one way to think of derivators is as some kind of weakly final approach to abstract homotopy theories. Quillen model categories and complete and cocomplete $\infty$-categories have underlying homotopy derivators (see \cite{cisinski:direct} in the first case and \cite{gps:mayer} for a sketch proof in the second case). And it is expected that there are variants of these results for other axiomatizations of $(\infty,1)$-categories. Conjecturally, these assignments preserve stable homotopy theories and exact morphisms of stable homotopy theories. Hence, once these transitions are understood in more detail, the results of this section have implications for these other approaches as well.
\end{rmk}

\bibliographystyle{alpha}
\bibliography{canon}

\end{document}